\documentclass[12pt,a4paper,reqno]{amsart} 
\usepackage[margin=1 in]{geometry}
\usepackage{setspace}

\usepackage[utf8]{inputenc}
\usepackage[titletoc]{appendix}
\usepackage{amsfonts,amssymb,amscd,graphicx, amsmath,hyperref,subcaption,lscape}
\usepackage{fancybox}
\usepackage{color,float} 
\usepackage{bbm}
\usepackage{cite} 

\usepackage{kbordermatrix}
\usepackage{wrapfig}
\allowdisplaybreaks
\usepackage{dsfont}
\usepackage[english]{babel}

\usepackage{tikz, tikz-cd}
\usetikzlibrary{snakes}
\usetikzlibrary{math}
\usepackage {tikz}
\usetikzlibrary {positioning}
\usepackage{setspace}
\usepackage{caption}

\usepackage{mathrsfs}
\usepackage[shortlabels]{enumitem}
\usepackage{xcolor}
\usepackage{tikz}
\mathchardef\ordinarycolon\mathcode`\:
\mathcode`\:=\string"8000
\begingroup \catcode`\:=\active
\gdef:{\mathrel{\mathop\ordinarycolon}}
\endgroup
\makeatletter
\renewcommand*\env@matrix[1][\arraystretch]{%
	\edef\arraystretch{#1}%
	\hskip -\arraycolsep
	\let\@ifnextchar\new@ifnextchar
	\array{*\c@MaxMatrixCols c}}
\makeatother

\usepackage{graphicx}
\usepackage{amsthm}
\usepackage{rotating}
\newtheorem{theorem}{Theorem}[section]

\newtheorem{corollary}{Corollary}[theorem]
\newtheorem{lemma}[theorem]{Lemma}

\theoremstyle{definition}
\newtheorem{definition}[theorem]{Definition}
\newtheorem{remark}[theorem]{Remark}
\newtheorem{exam}[theorem]{Example}
\newcommand{\F}{\mathcal{F}}
\newcommand{\G}{\mathcal{G}}

\newcommand{\cP}{\mathcal{P}}

\newcommand{\B}{\mathcal{B}}

\allowdisplaybreaks

\title[]{On Escape rate for subshift with Markov measure}
\author[Nikita Agarwal]{Nikita Agarwal}
\address{Department of Mathematics\\
	Indian Institute of Science Education and Research Bhopal\\
	Bhopal Bypass Road, Bhauri \\
	Bhopal 462 066, Madhya Pradesh\\
	India}
\email{nagarwal@iiserb.ac.in}
\author[Haritha Cheriyath]{Haritha Cheriyath}
\address{Centro de Modelamiento Matemático (CNRS IRL2807)\\Universidad de Chile\\Santiago, Chile}
\email{harithacheriyath@gmail.com}
\author[Sharvari Neetin Tikekar]{Sharvari Neetin Tikekar}
\address{School of Mathematics\\
	Tata Institute of Fundamental Research\\
	Mumbai 400 005, Maharashtra\\
	India}
\email{sharvari.tikekar@gmail.com}

\date{\today}

\begin{document} 
	\maketitle

	\begin{abstract} 
	In this paper, we present a precise formula to compute the escape rate into a hole in a subshift of finite type endowed with any Markov measure. The hole considered here is a set of sequences that do not begin with any of the words from a given finite collection. We obtain the escape rate directly in terms of the spectral radius of a perturbed stochastic matrix, where the perturbation rule is determined by the hole. By exploiting the combinatorial nature of the space under consideration and using the method of recurrence relations, we also establish that the escape rate into a hole is the logarithm of the smallest real pole of certain rational function. Both methods have their own merits, which are illustrated through examples.
	\end{abstract}
	

		\section{Introduction}
		The study of open dynamical systems, also known as dynamical systems with a hole, originated in 1979 in the work of Pianigiani and Yorke \cite{PY}. It is currently an active subbranch of dynamical systems and ergodic theory. 
		
		In a probability space $(X,\mathcal{B},\mu)$ with a measure preserving transformation $T$ on $X$, a hole $H$ is a measurable subset of $X$ having a positive measure. The orbit of a point $x \in X$ is said to \textit{escape into the hole $H$} if $T^k(x) \in H$ for some $k \ge 0$. That is, the orbit of $x$ intersects $H$. The {\it open system} is defined as $T|_{X \setminus H}: X \setminus H \rightarrow X$, and the original system $T : X \rightarrow X$ is referred to as the {\it closed system}. The measure of the set of points in $X$ which do not escape into the hole until time $n$ is called the {\it survival probability at time $n$}. The exponential rate of decay of these survival probabilities as $n \rightarrow \infty$ is known as the {\it escape rate into the hole $H$}. 
		
		Dynamical systems with holes represent systems on subsets of $X$, which are not invariant under the dynamics. Although it is a fairly recent area, some significant progress has been made towards the study of a variety of systems with holes. To cite a few, uniformly expanding or hyperbolic systems, including Smale horseshoes, Anosov diffeomorphisms~\cite{CMS:94,PY,Cencova} {\it etc.}, and also certain non-uniformly or weakly hyperbolic systems such as logistic maps~\cite{Demers:05} have been explored. In the probabilistic setting specifically, the finite and countable state Markov chains with holes~\cite{CMS:97,DIMMY} have been investigated. Many of these systems are conjugate to subshifts of finite type (SFTs). For example, it is well-known that certain piecewise linear maps and logistic maps on the unit interval are isomorphic to a full shift on an appropriate symbol set. Studying the dynamics via this conjugacy has proven to be a very powerful tool; hence, it makes sense to pose relevant questions in a more general framework of symbolic dynamics, which is the main theme of this paper.
		
		One of the natural questions that arises in the study of open dynamical systems is how the escape rate into a hole depends on the location and the size or measure of the hole, and to effectively classify holes according to certain parameters such as their measure, and identifying hole(s) of maximal escape rate in a given class. For the full shift over finite symbols with a single cylinder as the hole, this question was addressed by Bunimovich and Yurchenko in~\cite{BY} when the shift is equipped with the uniform probability measure. In general, if a hole is a finite union of cylinder sets in an irreducible SFT with the Parry measure, then the first two authors of this paper, in~\cite{HA:19}, provide a formula for escape rate into a hole as the difference between the topological entropy of the original SFT and that of the SFT consisting of sequences of the original SFT whose orbit under the left shift map does not intersect the hole. {\it The Parry measure} on an irreducible SFT is the unique invariant Markov measure of maximal entropy. Moreover, the authors compare escape rates into various holes having the same measure. In a follow-up work~\cite{HA:23}, they provide conditions on the size of the symbol set to obtain holes with large escape rates and establish a relationship between the escape rate and the minimum period of the hole. Their work focuses on the SFT with the Parry measure and a Markov hole. 
		
		Recently, in~\cite{BCL}, Bonanno {\it et. al.} consider SFTs with a single cylinder based at a word of finite length as a hole. They provide a complete characterization of such holes having maximal escape rate, when the SFT is equipped with a product measure. Comparing escape rates into different cylinders based at words of the same length, the authors prove that the escape rate is maximum either into a cylinder based at a prime word or a cylinder of maximum measure. However, the authors mention that the problem becomes much more intricate when the SFT is induced with a general Markov measure instead of the product measure. They discuss some results in the setting of Markov measure, however, only for shifts over the symbol set of size 2. Their work also sheds light on some difficulties in obtaining a precise expression for the escape rate into a Markov hole when the SFT is induced with a general Markov measure.
		
		In our work, we address the most general situation where the SFT is equipped with any Markov measure, and the hole is a finite union of cylinders. Our results subsume the two cases studied by Bonanno {\it et. al.}~\cite{BCL} for single cylinder as a hole in an SFT, first, when the SFT is equipped with product measure, and second, the SFT is on two symbols equipped with any Markov measure. We obtain an explicit formula for the escape rate in terms of the spectral radius of a certain matrix, which can be computed using higher block presentations of the shift and the stochastic matrix giving the Markov measure. This formula generalizes the one obtained in~\cite{HA:19} when the SFT is equipped with the Parry measure.  We also shed 
		light on some possible connections of our work with the work on perturbations of subshifts by Lind~\cite{Lind_pert} and Ramsey~\cite{Ramsey}, later in the paper. Due to the use of combinatorial techniques developed by Guibas and Odlyzko~\cite{Guibas}, the results bring out interesting observations on how the overlapping of words associated with the cylinders constituting the hole affects the escape rate.

		\subsection{Organization and Summary of main results}
		In Section \ref{sec:prelim}, we discuss preliminaries on subshifts of finite type in detail. Here, we only mention some of the necessary facts to the extent that help us state our main results in brief. 
		
		Let $\Sigma$ be a finite symbol set and $A$ be a $0-1$ matrix indexed by $\Sigma$ (we say that a \textit{matrix is indexed by a set} if its rows and columns are indexed by the elements of that set). The {\it subshift of finite type} (SFT) associated with $A$ is defined as the collection of sequences $\Sigma_A:= \left\lbrace x = (x_n)_{n \ge 1} \in \Sigma^{\mathbb{N}} \ | \ A_{x_ix_{i+1}} = 1, \ i \ge 1 \right\rbrace$. An \textit{allowed word} $u$ in $\Sigma_A$ is a finite string of symbols from $\Sigma$ which appears as a part of some sequence in $\Sigma_A$. A \textit{cylinder} based at an allowed word $u$, denoted by $C_u$, is defined as the set of all sequences in $\Sigma_A$ which begin with $u$. Let $\mu_P$ be the Markov measure on $\Sigma_A$ given by a stochastic matrix $P$.     
		
		This paper focuses on studying the escape rate into a Markov hole in $\Sigma_A$. By a Markov hole $H$, we mean a finite union of cylinders, that is, a hole $H=\cup_{u \in \G} C_u$, determined by a finite collection $\G$ of allowed words in $\Sigma_A$. In order to deal with the words of length larger than 2, we switch to the higher block representations of $\Sigma_A$ and $P$, as defined in Section~\ref{subsec:higher_block}. To avoid getting into the intricate nitty-gritty of the techniques involved, at this stage, let us assume that each word in $\G$ has length $2$. However, in this paper, we prove the results for when $\G$ consists of words of different lengths. Let $B$ (with size same as that of $A$, which is nothing but $|\Sigma|$) denote a $0-1$ matrix defined as follows: $B_{ij} = 0$, if and only if, either $A_{ij} = 0$, or the word $ij \in \G$. Note that $B$ is the adjacency matrix of the SFT consisting of those sequences in $\Sigma_A$ which do not contain words from $\G$. Let $\rho(H)$ denote the escape rate  into the hole $H$.

		
		In Section \ref{sec:esc_rate_form}, we obtain a precise formula for the escape rate into a Markov hole as stated below.
		\begin{enumerate}
			\item[(A)] {\bf Formula for the escape rate} (Theorem~\ref{thm:erformula}): The escape rate is given by $\rho(H) = -\ln \lambda(B\circ P)$, where $B \circ P$ is the Hadamard product or element-wise product of the matrices $B$ and $P$ and
			$\lambda(B \circ P)$ denotes its spectral radius. Moreover, $\rho(H)>0$.
		\end{enumerate}
		
		\noindent This formula can also be expressed as the difference between the topological pressure of certain potential on the original shift $\Sigma_A$ and that of the potential restricted to the new shift $\Sigma_{B}$, see Theorem \ref{thm:erate_pressure}. This formulation of escape rate generalizes the escape rate formula (as a difference of topological entropies) with respect to the Parry measure given in~\cite{HA:19}. See Section~\ref{subsec:er_and_pressure} for the concepts of pressure, potential {\it etc.} We further obtain an upper bound on the escape rate into a hole, which is a single cylinder, as follows:
		\begin{enumerate}
			\item[(B)] {\bf An upper bound on the escape rate when $H$ is a single cylinder} (Theorem~\ref{thm:location}): If $P$ has a positive real eigenvalue other than 1, and $\theta$ is the second largest real eigenvalue of $P$, then $\rho(H) \in (0, \ln \theta^{-1}]$.
		\end{enumerate}
		
		In Section~\ref{sec:rec_relations}, we obtain certain recurrence relations by using the weighted correlation polynomials of the words in $\G$, which gives another method to compute the escape rate. The techniques used to arrive at the recurrence relations are inspired by Guibas and Odlyzko~\cite{Combinatorial}. 
		\begin{enumerate}
			\item[(C1)] {\bf Escape rate as a logarithm of a pole of certain rational function} (Theorems~\ref{thm:esc_rate_as_smallest_pole} and~\ref{thm:rec_relations}): The escape rate into the Markov hole $H$ is the logarithm of the smallest positive real pole of certain rational function.
			\item[(C2)] {\bf Escape rate in terms of a root of a polynomial when $H$ is a single cylinder} (Theorem~\ref{thm:erone}): If $\G=\{u\}$ then the escape rate $\rho(H)$ is the logarithm of the smallest positive real root of certain polynomial, which is a function of the weighted correlation polynomial of $u$ and the stochastic matrix $P$.
		\end{enumerate}
		
		We also establish in Section~\ref{sec:rec_relations}, that the results in Bonanno~\textit{et. al.}~\cite{BCL} are special cases of our results. The techniques used in this paper are combinatorial in nature since we are dealing with words and sequences over a finite symbol set. In Section~\ref{sec:advantages}, we discuss the limitations and advantages of both the methods of computations of escape rate, (A) and (C1, C2). We conclude the paper with some numerical observations and derived conjectures. 
		

		\section{Preliminaries on symbolic dynamics}\label{sec:prelim}
		Let $\Sigma=\{1,2,\dots,N\}$ whose elements are known as \textit{symbols}. The set $\Sigma^{\mathbb{N}}$ consists of all one-sided sequences on $\Sigma$. For $\ell\ge 1$, each element of $\Sigma^\ell$ is known as a \textit{word} with length $\ell$ on $\Sigma$. A word $w=w_1w_2\dots w_n$ on $\Sigma$ is a \textit{subword} of a word $u=u_1u_2\dots u_k$ on $\Sigma$ if $k\ge n$ and there exists a $0\le j\le k-n$ such that $w_i=u_{j+i}$ for all $1\le i\le n$.
		A word $w=w_1w_2\dots w_n$ on $\Sigma$ \textit{appears in a sequence} $x=x_1x_2\dots\in\Sigma^\mathbb{N}$ if there exists $j\ge 1$ such that $w_i=x_{j+i}$ for all $1\le i\le n$. The function $\sigma:\Sigma^{\mathbb{N}}\to\Sigma^{\mathbb{N}}$ defined as $\sigma((x_n)_{n\in\mathbb{N}})=(x_{n+1})_{n\in\mathbb{N}}$, is known as the \textit{left shift map}. The set $\Sigma^\mathbb{N}$ is a compact metric space, and each $\sigma$-invariant closed subspace of $\Sigma^\mathbb{N}$ is called a \emph{subshift}. If a word appears in a sequence in a subshift, it is said to be an \emph{allowed word}. A word that is not allowed is said to be a \emph{forbidden word}. 
		
		Let $A$ be a matrix with entries 0 or 1, indexed by $\Sigma$. A \emph{subshift of finite type (SFT) with respect to $A$}, denoted by $\Sigma_A$, is defined as the collection of sequences $x=(x_n)_{n\ge 1}\in \Sigma^\mathbb{N}$ such that $A_{x_n x_{n+1}}=1$, for all $n\ge 1$. The matrix $A$ associated with the SFT $\Sigma_{A}$ is called as the {\it adjacency matrix} of $\Sigma_A$. It is well-known that $\Sigma_A$ is closed and invariant under $\sigma$ and hence is a subshift. Note that $\Sigma_A$ consists of all sequences in $\Sigma^\mathbb{N}$ in which words of the type $ij$ for which $A_{ij}=0$ do not appear. In other words, $\Sigma_A$ is completely described by a finite set of words of length two that are forbidden.  We say that a non-negative matrix $A$ is \emph{irreducible}, if for $i,j\in\Sigma$, if there exists a $n=n(i,j)>0$ such that $A^n_{ij}>0$. The SFT $\Sigma_A$ is said to be \emph{irreducible} if $A$ is irreducible. The SFT $\Sigma_A$ can be interpreted in terms of certain graph. Let $G_A=(V_A,E_A)$ be the directed graph with adjacency matrix $A$, where $V_A=\Sigma$ is the set of vertices, and $E_A$ is the set of all edges. An edge from a vertex $i$ to $j$ is labeled as $ij$, and $ij\in E_A$ if only if $A_{ij}=1$. Thus, a one-to-one correspondence exists between $\Sigma_A$ and one-sided infinite paths on $G_A$.  
		
		Let $\Sigma_A$ be an SFT. For each $n\ge 1$, let $\mathcal{L}_n$ denote the collection of all allowed words of length $n$ in the subshift $\Sigma_A$. 
		For $w\in \mathcal{L}_n$, the collection of all sequences in $\Sigma_A$ beginning with the word $w$ is known as a \emph{cylinder}, denoted by $C_w$. Further, for $X=x_1\dots x_n$, $Y=y_1\dots y_n\in\mathcal{L}_n$, two words of length $n$, with $x_2\dots x_{n}= y_1\dots y_{n-1}$, we define a word $X*Y$ of length $n+1$ as $X*Y:=x_1x_2\dots x_ny_{n}$. In such case, we say that $X$ and $Y$ \emph{overlap progressively}. For $n=1$, we define $X*Y=x_1y_1$.
		
		\subsection{Higher block representation of $\Sigma_A$}\label{subsec:higher_block}
		For $n\ge 1$, define a new subshift $\Sigma_A^{(n)}=\Sigma_{A_n}$, where $A_n$ is a $0-1$ matrix with rows and columns indexed by $\mathcal{L}_{n}$ (allowed words of length $n$ in $\Sigma_A$) given as follows: Set $A_1=A$ and $\Sigma_A^{(1)}=\Sigma_{A}$. Further, if $n \ge 2$, then for $X=x_1\dots x_{n}$ and $Y=y_1\dots y_{n}\in \mathcal{L}_{n}$,
		\[
		(A_n)_{XY} :=\begin{cases}
			1, & \text{ if }  x_2\dots x_{n}= y_1\dots y_{n-1} \text{ and } X*Y\in \mathcal{L}_{n+1},\\
			0, & \text{otherwise}.
		\end{cases}
		\]
		
		There is a one-to-one correspondence between the original subshift $\Sigma_A$ and the new subshift $\Sigma_{A}^{(n)}$ via  $x_1x_2x_3\dots \ \leftrightarrow \ (x_1\dots x_{n})(x_2\dots x_{n+1})(x_3\dots x_{n+2})\dots$. 
		
		The shift $\Sigma_{A_n}$ is known as the \textit{higher block representation of $\Sigma_A$ of order $n$}. Let $G_n = (V_n, E_n)$ denote the graph with adjacency matrix $A_n$. Here $V_n = \mathcal{L}_n$ and an edge from the vertex $X$ to $Y$, whenever it exists, is labeled as $X*Y \in \mathcal{L}_{n+1}$. The following example illustrates the higher block representation of an SFT.
		
		\begin{exam}
			Consider a subshift $\Sigma_A$ with $\Sigma=\{1,2\}$ and $A=\begin{pmatrix}
				0 & 1\\ 1 & 1
			\end{pmatrix}$. Take for instance $n=3$. The higher block presentation $\Sigma^{(3)}_A=\Sigma_{A_3}$ of $\Sigma_A$ of order 3 is given as follows. The set of words of length 3 that appear in $\Sigma_A$ is $\mathcal{L}_3=\{121,122,212,221,222\}$. Hence the matrix $A_3$ indexed by elements of $\mathcal{L}_3$ is given by\renewcommand{\kbldelim}{(}
			\renewcommand{\kbrdelim}{)}
			\[ 
			A_3 =  \kbordermatrix{
				&  121 & 122 & 212 & 221& 222\\
				121&  0 & 0 & 1 & 0& 0\\
				122&  0 & 0 & 0 & 1& 1\\
				212&  1 & 1 & 0 & 0& 0\\
				221&  0 & 0 & 1 & 0& 0\\
				222&  0 & 0 & 0 & 1& 1}.
			\]
			Note that $(121)(121)$ entry of $A_3$ is 0 since $121$ and $121$ do not progressively overlap, and others are 0 for the same reason. Figures~\ref{fig:1} and~\ref{fig:2} show the graphs $G_A^{(1)}=G_A,$ and $G_A^{(3)}$ respectively. In these figures an edge from vertex $X$ to vertex $Y$ is labeled by the word $X*Y$.
			
			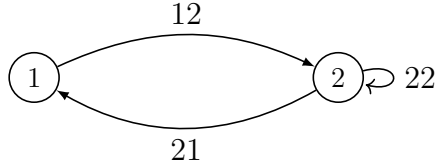
\begin{figure}[h]
				\centering
				$\displaystyle
				\begin {tikzpicture}[-latex ,auto ,node distance =3cm and 4cm ,on grid , semithick , state/.style ={draw, circle}] 
				\node[state,scale=0.9] (A) {$1$};
				\node[state,scale=0.9] (B) [right =of A] {$2$};
				\path (A) edge [bend left =25] node[above] {$12$} (B);
				\path (B) edge [bend right = -30] node[below] {$21$}(A);
				\path (B) edge [loop right] node[right] {$22$} (B);
			\end{tikzpicture}
			$
			\caption{The graph $G_A$ corresponding to the shift $\Sigma_A$}
			\label{fig:1}
		\end{figure}
		
		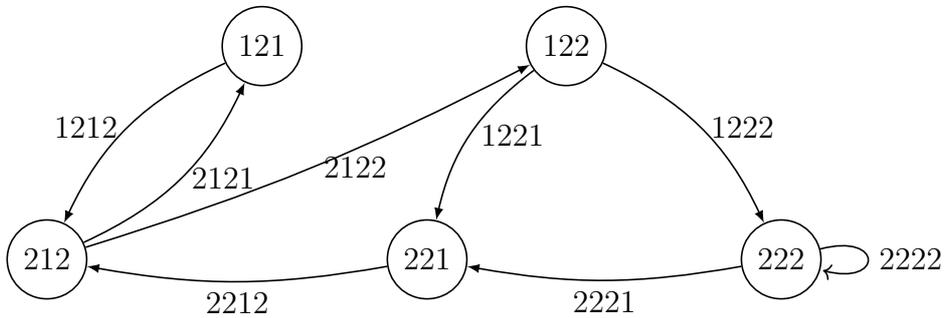
\begin{figure}[h]
			\centering
			$\displaystyle
			\begin {tikzpicture}[-latex, auto ,node distance =4cm, ,on grid, semithick, state/.style ={draw, circle}]
			\node[state] (A) {$121$};
			\node[state] (B) [right =of A] {$122$};
			\node[state] (C) [below left of = A] {$212$};
			\node[state] (D) [right = 5of C] {$221$};
			\node[state] (E) [below right of = B] {$222$};
			\path (A) edge [bend right =20] node[left] {$1212$} (C);
			\path (C) edge [bend right =20] node[right] {$2121$} (A);
			\path (D) edge [bend left =10] node[below] {$2212$} (C);
			\path (E) edge [bend left =10] node[below] {$2221$} (D);
			\path (B) edge [bend right =20] node[right] {$1221$} (D);
			\path (B) edge [bend left =20] node[right] {$1222$} (E);
			\path (E) edge [loop right] node[right] {$2222$} (E);
			\path (C) edge [bend right= 5] node[right]{$2122$} (B);
		\end{tikzpicture} $
		\caption{The graph $G_{A_3}$ corresponding to the higher block presentation of $\Sigma_A$ of order 3}
		\label{fig:2}
	\end{figure}
\end{exam}



\subsection{Markov measure on a SFT}
Let $\Sigma=\{1,2,\dots,N\}$ and let $A$ be an irreducible matrix with entries 0 or 1. Let $P=(P_{ij})$ be a row stochastic matrix of size $N$ which is \textit{compatible with $A$}. That is, $P_{ij}> 0$ if and only if $A_{ij}=1$, for each $i,j\in \Sigma$. A row stochastic matrix is a non-negative matrix having every row sum to be $1$. We now define a measure on $\Sigma_A$ with respect to $P$, denoted by $\mu_P$, known as a \textit{Markov measure}. We know that $1$ is the Perron root of $P$ (which is the largest positive eigenvalue of $A$), and the column vector of size $N$ denoted by $\mathbf{1}$ having all entries $1$, is the corresponding right Perron eigenvector of $P$. Let $\textbf{p}=(p_1,\dots,p_N)^T$ be a normalized left Perron column eigenvector of $P$ corresponding to Perron root 1, that is, $\textbf{p}^TP=\textbf{p}^T$ and $p_1+\dots+p_N=1$. The vector $\textbf{p}$ is known as the \textit{stationary vector}. Since $A$ is irreducible, the matrix $P$ is also irreducible. Hence, the stationary vector $\textbf{p}$ exists and is unique by the Perron-Frobenius Theorem.

Let $C_w$ be the cylinder based at an allowed word $w=w_1\dots w_n$ in $\Sigma_A$. Then $\mu_P$ is defined on cylinders as $\mu_P(C_w) := p_{w_1}P_{w_1w_2}\dots P_{w_{n-1}w_n}$. If $w$ is just a symbol, that is $w=w_1$, then $\mu(C_w) := p_{w_1}$. By the Kolmogorov Extension Theorem, this uniquely
defines a measure $\mu_P$ on the Borel sigma-algebra generated by cylinders. When the context is clear, we use the notation $\mu$ instead of $\mu_P$. The measure $\mu_P$ is a probability measure on $\Sigma_A$ that is invariant under $\sigma$. We use shorthand $\mu(w)$ to denote the measure $\mu_P(C_w)$ of $C_w$. Due to the irreducibility of $A$, the SFT $\Sigma_A$ is ergodic with respect to $\mu_P$. 

The Markov measure, which is the unique measure of maximal entropy on $\Sigma_A$, is known as the \textit{Parry measure}. It is defined as follows. Let $\textbf{u}=(u_1,\dots,u_N)^T$, $\textbf{v}=(v_1,\dots,v_N)^T$ be the positive left and right Perron eigenvectors of the matrix $A$ corresponding to the Perron root $\lambda$ (which is the largest positive eigenvalue of $A$ that exists by the Perron-Frobenius theorem). Define a stochastic matrix $P$ as follows. For each $i,j\in \Sigma$, let $P_{ij}=\dfrac{A_{ij}v_j}{\lambda v_i}$, $p_i=\dfrac{u_iv_i}{\textbf{u}^T\textbf{v}}$. The Parry measure is $\mu=\mu_P$ with $P$ just defined.

\subsection{Higher block representation of the Markov measure $\mu_P$}
We aim to describe the measure on $\Sigma_A^{(n)}$ induced by the Markov measure $\mu_P$ on $\Sigma_A$. The induced measure is also a Markov measure. For each $n\ge 1$, define a matrix $P_n$ indexed by $\mathcal{L}_n$. For each $X=x_1\dots x_n$ and $Y=y_1\dots y_n$ in $\mathcal{L}_n$, define $(P_n)_{XY}=P_{x_ny_n}$, if $X*Y\in \mathcal{L}_{n+1}$, and is 0, otherwise. 

Note that $P_1=P$. It is easy to check that $P_n$ is a stochastic matrix compatible with $A_n$, and its stationary vector is given by $\mathbf{\Lambda}_n=(\mu(X))_{X\in \mathcal{L}_n}$. That is, $\mathbf{\Lambda}_n^T P_n=\mathbf{\Lambda}_n^T$. The measure induced by the Markov measure $\mu_P$ on the higher block presentation $\Sigma_{A}^{(n)}$ of $\Sigma_A$ is nothing but the Markov measure $\mu_{P_n}$. 


\begin{exam}
Consider the SFT $\Sigma_A$ with $\Sigma=\{1,2\}$ and $A=\begin{pmatrix}
	0 & 1\\ 1 & 1
\end{pmatrix}$. Any stochastic matrix $P=(P_{ij})_{1\le i,j\le 2}$ compatible with $A$ satisfies $P_{11}=0$, hence $P_{12}=1$ since $P$ is row stochastic. For $n=2$, the matrix $P_2$ has size 3 indexed by words from $\mathcal{L}_2=\{12,21,22\}$. We have 
\renewcommand{\kbldelim}{(}
\renewcommand{\kbrdelim}{)}
\[ 
A_2 =  \kbordermatrix{
	&  12 & 21 & 22\\
	12 & 0 & 1 & 1 \\
	21 & 1 & 0 & 0\\
	22 & 0 & 1 & 1 
}, \  
P_2 =  \kbordermatrix{
	&  12 & 21 & 22\\
	12 & 0 & P_{21} & P_{22} \\
	21 & 1 & 0 & 0\\
	22 & 0 & P_{21} & P_{22}
}.
\]
\end{exam}

\section{Escape rate formula using higher block representation of the Markov measure}\label{sec:esc_rate_form}
In this section, we consider an SFT $\Sigma_A$ with a Markov measure $\mu_P$. We will first recall the definition of escape rate into a hole in $\Sigma_A$. We then state and prove one of the main results of our paper, which establishes a direct formula for the escape rate, see Theorem~\ref{thm:erformula}. Further, Theorem~\ref{thm:location} gives bounds on the value of the escape rate into a hole, which is a single cylinder. 

\begin{definition}[Markov Hole]
Let $\G$ be a finite collection of allowed words in an SFT $\Sigma_A$. A \textit{Markov hole given by the collection $\G$}, denoted by $H_\G$, is defined as the union of cylinders based at words from $\G$. That is, $H_\G = \bigcup_{w\in\G}C_w$.
\end{definition}

In the rest of the paper, $\G$ denotes a finite collection of allowed words in $\Sigma_A$ and will define the Markov hole.

\begin{definition}[Escape rate into a Markov hole]
Let $\Sigma_A$ be an SFT with a Markov measure $\mu=\mu_P$. The escape rate into the hole $H_\G$ is defined as 
\begin{equation}\label{eq:escrate_defn}
	\rho(H_\G) :=-\lim_{m\to\infty}\frac{1}{m}\ln\mu(\mathcal{W}_m),
\end{equation}
provided the limit exists, where $\mathcal{W}_{m}=\{x\in\Sigma_A\ : \ \sigma^ix\notin H_\G, \text{ for } 0\le i\le m\}$.
\end{definition}

\noindent In due course, we will prove that the above limit exists and give two approaches to compute the escape rate.

\subsection{Adjacency matrix of the open system}
Let $\Sigma_A$ be an SFT with a Markov hole $H_\G$. Let $r$ be the length of the longest word in $\G$. We now define a matrix $B = B_\G$ with each of its entries either 0 or 1, as follows. For $r \ge 2$, the matrix $B$ is indexed by $\mathcal{L}_{r-1}$, and for each $X,Y\in\mathcal{L}_{r-1}$, $B_{XY} = 1$ if and only if $X*Y$ belongs to $\mathcal{L}_r$ and does not contain a word from $\G$. For $r=1$, we fix the convention that, $B$ is a matrix indexed by $\Sigma$, defined in the same fashion above.

The matrix $B_\G$ will be called the \textit{adjacency matrix of the open system} throughout the rest of the paper. The SFT $\Sigma_{B_\G}$ is the collection of all those sequences in $\Sigma_A$ which do not contain words from $\G$. The matrix $B_\G$ is the adjacency matrix of the graph obtained from $G_{A_{r-1}}$ by removing all those edges whose labels contain a word from $\G$.

\begin{exam}
	Let $\Sigma=\{1,2\}$, $A=\begin{pmatrix}
		0 & 1\\ 1 & 1
	\end{pmatrix}$, $\mathcal{G}=\{121,1221\}$. Since the longest word in $\mathcal{G}$ has length 4, the matrix $B_{\mathcal{G}}$ is indexed by words in $\mathcal{L}_3=\{121,122,212,221,222\}$. The adjacency matrix corresponding to the hole $H_\G$ is given by 
	\[
	B_{\mathcal{G}} = \renewcommand{\kbldelim}{(}
	\renewcommand{\kbrdelim}{)}
	\kbordermatrix{
		& 121 & 122 & 212 & 221 & 222\\
		121 & 0 & 0 & \textbf{0}  & 0 & 0\\
		122& 0 & 0 & 0 & \textbf{0} & 1\\
		212 & \textbf{0} & 1 & 0 & 0 & 0\\
		221 & 0 & 0 & 1 & 0 & 0\\
		222 & 0 & 0 & 0 & 1 & 1
	}.
	\]
	The $0$ entries in boldface correspond to $X=x_1x_2x_3, Y=y_1y_2y_3\in \mathcal{L}_3$ for which $x_2x_3=y_1y_2$ (progressive overlap), but $X*Y=x_1x_2x_3y_3$ contains a word from $\G$.
\end{exam}


\subsection*{Formula for the escape rate}

The following result gives a formula to compute the escape rate into a Markov hole. For any two matrices $A$ and $B$ of the same size, we denote the {\it Hadamard product}, or entrywise product of $A$ and $B$ by $A \circ B$. 

\begin{theorem}\label{thm:erformula} (With the notations as above) Let $\Sigma_A$ be an SFT with a Markov measure $\mu_P$. Consider the Markov hole $H_\G$ given by a collection $\G$. Let $r$ denote the length of the longest word in $\G$. The escape rate into the hole $H_\G$ is given by,
	\[
	\rho(H_\G) = 
	-\ln \lambda(B_\G\circ P_{r-1}),
	\]
	where $\lambda(B_\G\circ P_{r-1})$ denotes the spectral radius of the Hadamard product of $B_\G$ and $P_{r-1}$. Moreover, $\rho(H_\G)>0$.
\end{theorem}

\begin{proof} 
	For simplicity of notation, we will use shorthand $B$ for $B_\G$. Consider $\mathcal{W}_m$, the collection of all sequences in $\Sigma_A$ for which $\sigma^ix\notin H_\G$, for all $0\le i\le m$. Let $x=(x_n)_{n\ge 1}\in \mathcal{W}_m$. For each $i\ge 1$, $X_i=x_i\dots x_{i+r-2}\in \mathcal{L}_{r-1}$, and none of the words $X_1,\dots,X_{m+1}$ begin with a word from $\G$. Hence 
	\[
	\mu(\mathcal{W}_m)  =  \sum_{\substack{X_1, \dots, X_{m+1} \in \mathcal{L}_{r-1}, \\ \text{none beginning with a word from }\G}} \mu_{P_{r-1}}(C_{X_1\dots X_{m+1}}).
	\]
	Moreover, none of the words $X_1,\dots,X_{m-r+2}$ contain a word from $\G$. Hence 
	\begin{eqnarray}\label{eq:S1S2}
		S_1 \le \mu(\mathcal{W}_m)\le S_2,
	\end{eqnarray}
	where 
	\begin{eqnarray*}
		S_1&=&\sum_{\substack{X_1, \dots, X_{m+1} \in \mathcal{L}_{r-1}, \\ \text{none containing a word from }\G}} \mu_{P_{r-1}}(C_{X_1\dots X_{m+1}}),\\
		S_2&=&\sum_{\substack{X_1, \dots, X_{m-r+2} \in \mathcal{L}_{r-1}, \\ \text{none containing a word from }\G}} \mu_{P_{r-1}}(C_{X_1\dots X_{m-r+2}}).
	\end{eqnarray*}
	Now observe that 
	\begin{eqnarray*}
		S_1&=&\sum_{X_1, \dots, X_{m+1} \in \mathcal{L}_{r-1}} \mathbf{\Lambda}_{X_1} (B\circ P_{r-1})_{X_1X_2}\dots (B\circ P_{r-1})_{X_mX_{m+1}}\\
		&=&\sum_{X_1\in \mathcal{L}_{r-1}}\mathbf{\Lambda}_{X_1}R_{X_1}^{m},
	\end{eqnarray*}
	where $\mathbf{\Lambda}$ is the left Perron eigenvector of $P_{r-1}$ and $R^{m}_{X_1}$ denotes the $X_1^{th}$ row sum of $(B \circ P_{r-1})^{m}$. Similarly,
	\[
	S_2=\sum_{X_1\in \mathcal{L}_{r-1}}\mathbf{\Lambda}_{X_1}R_{X_1}^{m-r+1}.
	\]
	Let $\mathbf{\Lambda}_{min}$ and $\mathbf{\Lambda}_{max}$ denote the smallest and the largest entry of $\mathbf{\Lambda}$. Both are positive since $P$, hence $P_{r-1}$, is irreducible. Thus
	\[
	\mathbf{\Lambda}_{min} \sum (B\circ P_{r-1})^{m} \le S_1 \le \mathbf{\Lambda}_{max} \sum (B\circ P_{r-1})^{m},
	\]
	and 
	\[
	\mathbf{\Lambda}_{min} \sum (B\circ P_{r-1})^{m-r+1} \le S_2 \le \mathbf{\Lambda}_{max} \sum (B\circ P_{r-1})^{m-r+1},
	\]
	where $\sum (B\circ P_{r-1})^{n}$ denotes the sum of all the entries of the matrix $(B\circ P_{r-1})^{n}$. By~\eqref{eq:S1S2},
	\begin{eqnarray}\label{eq:erate-BoP}
		\mathbf{\Lambda}_{min} \sum (B\circ P_{r-1})^{m} \le \mu(\mathcal{W}_m)\le \mathbf{\Lambda}_{max} \sum (B\circ P_{r-1})^{m-r+1}.
	\end{eqnarray}
	Taking logarithm and dividing by $m$ on all terms and then taking the limit as $m \to \infty$, we get that the limit $ \lim\limits_{m\rightarrow \infty} \frac{1}{m}  \ln \mu(\mathcal{W}_m) $ exists and is given by,
	\[
	\rho(H_{\G}) = - \lim_{m\rightarrow \infty} \frac{1}{m}  \ln \sum (B\circ P_{r-1})^m.
	\]
	Since $B \circ P_{r-1}$ is non-negative, this is further equal to $-\ln \lambda(B\circ P_{r-1})$.\\
	Finally, since $P_{r-1}$ is a stochastic matrix and $0<B\circ P_{r-1}<P_{r-1}$, we have $0\le \lambda(B\circ P_{r-1})<1$. Hence  $\rho(H_{\G})>0$. 
\end{proof}

\begin{corollary}\label{cor:parry}
	Let $\mu = \mu_P$ be the Parry measure on $\Sigma_A$. The escape rate into the hole $H_\G$ is given by,
	\begin{equation*}
		\rho(H_\G)=-\ln\left(\frac{\lambda(B_\G)}{\lambda(A)}\right)=h_{\text{top}}\left(\Sigma_{A}\right)-h_{\text{top}}\left( \Sigma_{B_\G} \right),
	\end{equation*}
	where $h_{\text{top}}$ is the topological entropy of the SFT.
\end{corollary}

The above corollary follows from~\cite[Theorem 3.1]{HA:19} directly. We give a direct proof when all the words in $\G$ are of length 2.

\begin{proof}
	Let $\theta$ be an eigenvalue of $B_\G$ with eigenvector $w=(w_1,\dots,w_N)$. Then 
	\[\sum_j(B_\G\circ P)_{ij}\frac{w_j}{v_j}=\sum_j(B_\G)_{ij}P_{ij}\frac{w_j}{v_j}=\sum_j(B_\G)_{ij}w_j\frac{A_{ij}}{\lambda(A)v_i}=\frac{\theta}{\lambda(A)}\frac{w_i}{v_i}.
	\]
	We have used a simple observation that $(B_\G)_{ij}A_{ij}=(B_\G)_{ij}$. Hence $\lambda(B_\G\circ P)=\frac{\lambda(B_\G)}{\lambda(A)}$. The result follows.
\end{proof}

\begin{exam}
	Consider the full shift on $\Sigma=\{1,2\}$. That is, $\Sigma=\Sigma_A$ where $A$ is a matrix of size 2 having all its entries 1. Consider the Markov measure on $\Sigma_A$ given by the stochastic matrix $P$. First we consider a single cylinder based at a word of length 2 as the hole. There are 4 such holes. Table~\ref{tab:length2} shows the spectral radius calculations.
	
	\begin{table}[h!]
		\centering
		\begin{tabular}{|c|c|c|}
			\hline
			$\G$ & $B_\G$ & $\lambda(B_\G\circ P)$  \\
			\hline \hline
			$\{11\}$ & $\begin{pmatrix}
				0 & 1\\ 1& 1
			\end{pmatrix}$ & $\dfrac{P_{22} + \sqrt{P_{22}^2+4P_{12}P_{21}}}{2}$\\ 
			\hline
			$\{22\}$ & $\begin{pmatrix}
				1 & 1\\ 1& 0
			\end{pmatrix}$ & $\dfrac{P_{11} + \sqrt{P_{11}^2+4P_{12}P_{21}}}{2}$\\
			\hline
			$\{12\}$ & $\begin{pmatrix}
				1 & 0\\ 1& 1
			\end{pmatrix}$ & $\max\{P_{11}, P_{22}\}$\\
			\hline 
			$\{21\}$ & $\begin{pmatrix}
				1 & 1\\ 0& 1
			\end{pmatrix}$ & $\max\{P_{11}, P_{22}\}$  \\
			\hline
		\end{tabular}
		\caption{Spectral radius for $B_\G\circ P$ for $\G$ consisting of a single word}
		\label{tab:length2}
	\end{table}
	
	Now consider the hole given by the collection $\G=\{11,122\}$. Since the length of the longest word of $\G$ is 3, we need to compute the following matrices of size $|\mathcal{L}_2|=4$: $B_\G$ and $P_2$, indexed by words $11,12,21,22$ of $\mathcal{L}_2$ (allowed words of length 2). Here 
	\renewcommand{\kbldelim}{(}
	\renewcommand{\kbrdelim}{)}
	\[
	B_\G = \kbordermatrix{
		& 11 & 12 & 21 & 22\\
		11&\textbf{0} & \textbf{0} & 0 & 0\\
		12&0 & 0 & 1 & \textbf{0}\\
		21&\textbf{0} & 1 & 0 & 0\\
		22&0 & 0 & 1 & 1
	}, \ \ P_2 =   \kbordermatrix{
		& 11 & 12 & 21 & 22\\
		11&P_{11} & P_{12} & 0 & 0\\
		12&0 & 0 & P_{21} & P_{22} \\
		21&P_{11} & P_{12} & 0 & 0\\
		22&0 & 0 & P_{21} & P_{22} 
	}.
	\]
	Hence $\lambda(B_\G\circ P_2)=\max\left\lbrace P_{22},\sqrt{P_{12}P_{21}}\right\rbrace$.
\end{exam}


\subsection{An upper bound for the escape rate}
Our next main result (Theorem~\ref{thm:location}) gives an upper bound for the escape rate, in the special case when $\G$ contains a single word. In other words, we look for a lower bound for $\lambda(B_{\G} \circ P_{r-1})$. We may think of the matrix $B_{\G} \circ P_{r-1}$ as a perturbation of $P_{r-1}$ by changing one of its positive entries to $0$. To prove this theorem, we recall an important result on the separation of eigenvalues of a non-negative matrix and its principal submatrices, due to Hall and Porsching.

Due to the Perron-Frobenius Theorem, we know that the largest positive eigenvalue of a non-negative matrix $M$ equals its spectral radius, denoted by $\lambda(M)$. Let $M(i)$ denote the principal submatrix of $M$ obtained by removing the $i^{th}$ row and the $i^{th}$ column of $M$.

\begin{theorem}{\cite{Hall}}
	\label{thm:interlacing_non-negative}
	Let $M$ be a non-negative matrix of size $m \ge 2$ having a real eigenvalue other than its spectral radius $\lambda(M)$. Let $\theta$ be a real eigenvalue of $M$ other than $\lambda(M)$. Then $\theta \le \lambda(M(i)) \le \lambda(M)$, for all $1 \le i \le m$. The second inequality is strict if $M$ is irreducible. 
\end{theorem}

We now proceed to the following result, which gives an upper bound for the escape rate into a hole, which is a single cylinder.

\begin{theorem}\label{thm:location}
	Let $\Sigma_A$ be an SFT with a Markov measure $\mu_P$. Suppose $P$ has a positive real eigenvalue other than 1, and let $\theta$ denote its second largest positive real eigenvalue. Then the escape rate $\rho(C_u)$ into the cylinder $C_u$ satisfies $0<\rho(C_u)\le -\ln\theta$. 
\end{theorem}

\begin{proof}
	Let us first assume that the length of $u$ is 1 and $u=i$. Then $(B_\G\circ P)(i)=P(i)$. Hence by Theorem~\ref{thm:interlacing_non-negative},
	\[
	\theta\le \lambda(P(i))=\lambda((B_\G\circ P)(i))\le \lambda(B_\G\circ P)< 1.
	\]
	Now, consider the general case. Let $u=u_1\dots u_r$. Consider the subword $w=u_1$ of $u$. By the above arguments, $0<\rho(C_{w})\le -\ln\theta$. Further, since $C_u\subseteq C_{w}$, we have $\rho(C_u)\le \rho(C_{w})$, and the result follows.
	%
	%
\end{proof}

\begin{corollary}
	Suppose $P$ has a positive real eigenvalue other than 1, let $\theta$ denote its second largest positive eigenvalue. Let all the words in $\G$ have at least one common symbol. Then $0<\rho(H_{\G})\le -\ln\theta$.    
\end{corollary}
\begin{proof}
	Let $i$ be the common symbol in all the words in $\G$. If $u \in \G \cap \Sigma$, then clearly $u=i$ and set $u'=i$. For each $u\in \G \setminus \Sigma$, let $u'$ be a subword of $u$ of length 2 containing the common symbol $i$. Let $H_\G'=\bigcup_{u\in\G} C_{u'}$. Since $\mathcal{W}_m(H_\G')\subseteq \mathcal{W}_m(H_\G)$ for $m\ge 1$, we have $\rho(H_\G)\le \rho(H_\G')$. Now $(B_{H_\G'}\circ P)(i)=P(i)$. The result follows by similar arguments as in the proof of Theorem~\ref{thm:location}.
\end{proof}

It is obvious that the escape rate need not be bounded above in general. For instance, when $\G$ consists of all the cylinders of length 2, then $H_\G=\Sigma_A$, in which case, the escape rate is $\infty$. We also give an example when $\G$ has two words, and the escape rate is not bounded above by $-\ln \theta$.

\begin{exam}
	Let $\Sigma=\{1,2,3\}$ and \[P=\begin{pmatrix}
		1/5&2/5&2/5\\
		9/10&1/10&0\\
		1/10&1/10&4/5
	\end{pmatrix},\] and $\G=\{21,33\}$. Here, $\theta=\frac{1}{20} (1 + \sqrt{97})>\frac{1}{10} (1 + \sqrt{5})=\lambda(B_\G\circ P)$. Hence, $\rho(H_\G)>-\ln(\theta)$. 
\end{exam}

	
	\begin{remark}\label{rem:length2}
		When $r=1$, $\G$ consists only of symbols, and $B_\G$ is the matrix indexed by $\Sigma$ such that $B_{ij} = 1$ if and only if $ij \in \mathcal{L}_2$ and, $i,j \notin \G$. Construct a new collection of words of length 2 given by $\G' = \{ w  \in \mathcal{L}_2 \ : \ w \text{ contains } i, \ \text{for some } i  \in  \G \}$.
		If $r \ge 2$, then construct $\G' = \{ w  \in \mathcal{L}_r \ : \ w \text{ contains } u, \ \text{for some } u  \in  \G \}$ and switch to higher block representation of $\Sigma_A$ of order $r-1$. Hence for all $r\ge 1$, by definition, $B_{G} = B_{\G'}$ and $\rho(H_{\G}) = \rho(H_{\G'})$. This implies, to study the escape rate, one can as well assume all words in $\G$ to be of length $2$. This fact is mainly useful for the qualitative study of escape rate as evident in the following section.
	\end{remark}

	\subsection{Escape rate and thermodynamic formalism}\label{subsec:er_and_pressure}
	Thermodynamic formalism is an essential tool in ergodic theory and specifically to study shift spaces. In this section, we obtain another expression for the formula obtained for the escape rate in Theorem~\ref{thm:erformula} in terms of pressure. Let us first quickly recall some key concepts from thermodynamic formalism. We refer the reader to~\cite{Parry_pollicott_ast,Parry_Tuncel} for a comprehensive introduction to the subject.
	
	Let $\Sigma_A$ be an SFT. Let $\mathcal{C}(\Sigma_A)$ denote the space of real-valued continuous functions on $\Sigma_A$, and $\mathcal{M}_{\sigma}(\Sigma_A)$ denote the space of all $\sigma$-invariant probability measures on $\Sigma_A$. For a measure $\nu \in \mathcal{M}_{\sigma}(\Sigma_A)$, let $h_{\nu}(\sigma)$ denote the measure theoretic entropy of $\sigma$ with respect to $\nu$.
	We now give two basic definitions.
	\begin{itemize}
		\item {\bf Pressure and equilibrium states:} For $\phi \in \mathcal{C}(\Sigma_A)$, the {\it pressure of $\phi$}, denoted as $\cP_A(\phi)$, is defined as
		\[
		\cP_A(\phi):= \sup\left\{h_{\nu}(\sigma) + \int_{\Sigma_A} \phi \, d\nu \ : \ \nu \in \mathcal{M}_{\sigma}(\Sigma_A) \right\}.
		\]
		Each measure for which the above supremum is attained is called an {\it equilibrium state for $\phi$}.
		\item {\bf Transfer operator:} For any $\phi \in \mathcal{C}(\Sigma_A)$, the transfer operator $L_{\phi} : \mathcal{C}(\Sigma_A) \rightarrow \mathcal{C}(\Sigma_A)$ is defined as
		\[
		L_{\phi} g (x) := \sum\limits_{y \in \sigma^{-1}(x)} e^{\phi(y)} g(y).
		\]
		If $\phi$ is a function with {\it summable variation}, that is, \\
		$\sum_{k \ge 1} \sup \left\{ |\phi(x) - \phi(y)|\ : \ x_i = y_i \ \text{for all } i=1,\dots,k \right\} < \infty$, then the {\it Ruelle Perron-Frobenius Theorem} for transfer operator states that there exists a maximal simple positive eigenvalue $\lambda$ for $L_{\phi}$ and it is given by $\lambda = e^{\cP(\phi)}$. In most literature, $\phi$ is called {\it potential}.
	\end{itemize}
	
	\noindent Let us now view the formula obtained in Theorem~\ref{thm:erformula} using this formalism. Let $\G$ be a finite collection of allowed words from $\Sigma_A$. By virtue of Remark~\eqref{rem:length2}, without loss of generality, we assume that all words in $\G$ are of equal length $2$. Let $H_{\G}$ be the Markov hole given by the collection $\G$ and $B_{\G}$ be the adjacency matrix of the open system as defined earlier. Observe that $\sigma: \Sigma_A \rightarrow \Sigma_A$ denotes the {\it closed system} and $\sigma|_{B_{\G}}: \Sigma_A\setminus H_\G \rightarrow \Sigma_A$ denotes the {\it open system}. Consider a Markov measure $\mu_P$ on $\Sigma_A$. We first look at the case when $\mu_P$ is the Parry measure.
	
	It is well-known that the Parry measure $\mu$ on $\Sigma_A$ is the unique equilibrium state for the function $\phi \equiv 0 \in \mathcal{C}(\Sigma_A)$. That is, it is the unique measure of maximal entropy satisfying $\cP_A(0) = h_{\mu}(\sigma)$, which equals the topological entropy of $\Sigma_A$, denoted as $h_{\text{top}}\left(\Sigma_{A}\right)$. Moreover, for $\phi \in \mathcal{C}(\Sigma_A),$ let $\cP_{B_{\G}}(\phi)$ denote the pressure of $\phi|_{\Sigma_{B_\G}} \in \mathcal{C}(\Sigma_{B_\G})$, the restriction of $\phi$ to $\Sigma_{B_{\G}}$.
	Hence, the escape rate into the hole $H_\G$ with respect to the Parry measure $\mu$ (given in Corollary~\ref{cor:parry}) can also be expressed as
	\begin{eqnarray}\label{eq:thermo-parry} 
		\rho(H_\G)&=& \cP_A(0) - \cP_{B_{\G}}(0).
	\end{eqnarray}
	
	Consider the general case of a Markov measure $\mu_P$ associated with a stochastic matrix $P$ on $\Sigma_A$. It is known that $\mu_P$ is the unique equilibrium state for the function $\phi=\phi_P \in \mathcal{C}(\Sigma_A)$ given by $\phi(x) = \ln P_{x_1 x_2}$, for $x=(x_n)_{n\ge 1}\in \Sigma_A$. Also, 
	\[
	\phi|_{\Sigma_{B_\G}} (x)  = \ln (B_{\G} \circ P)_{x_1 x_2}=\ln P_{x_1x_2}, \ \text{for}\  x=(x_n)_{n\ge 1}\in \Sigma_{B_\G},
	\]
	since $x_1x_2 \notin \G$. With these notations in place, we have the following result.
	\begin{theorem}\label{thm:erate_pressure}
		Consider the SFT $\Sigma_A$ with Markov measure $\mu_P$. Let $\phi_P$ denote a potential of summable variation for which $\mu_P$ is the equilibrium state. The escape rate into the hole $H_{\G}$ is given by 
		\[
		\rho(H_{\G}) = \cP_A(\phi_P) - \cP_{B_\G}(\phi_P).
		\]
	\end{theorem}
	
	\begin{remark}
		A similar result is obtained by Tanaka~\cite[Theorem 3.7]{Tanaka} when the symbol set is countable. Moreover, a formula for the escape rate into a hole with respect to a Gibbs measure in terms of the spectral radius of the associated transfer operator is obtained by Ferguson and Pollicott~\cite[Proposition 5.2]{Ferg_polli}. Both works follow the approach of singular
		perturbations of transfer operators associated with the subshifts and H\"older continuous potentials with summable variation. Here, we give a direct proof specifically for our setting by exploiting the exact form of the Markov measure $\mu_P$ and the associated function $\phi_P$.
	\end{remark}
	
	\begin{proof}
		For ease of notation, we denote $\phi_P$ as just $\phi$. Since the measure theoretic entropy $h_{\mu_P}(\sigma) = -\int \phi \, d\mu$, we have $\cP_A(\phi) = 0$. Observe that the function $\phi$ depends only on the first two co-ordinates of a sequence, hence by abuse of notation, we denote $\phi(x) = \phi(x_1, x_2)$, where $x = (x_n)_{n\ge 1} \in \Sigma_A$. 
		
		
		\noindent We know that the transfer operator $L_{\phi|_{\Sigma_{B_\G}}}$ can be identified with a matrix $M$ indexed by elements of $\Sigma$, given by
		\[
		M_{x_1x_2} = (B_{\G})_{x_1x_2}e^{\phi|_{\Sigma_{B_\G}}(x_1,x_2)} = (B_{\G} \circ P)_{x_1x_2}.
		\]
		Hence the spectral radius of $L_{\phi|_{\Sigma_{B_\G}}}$ equals the spectral radius of the matrix $B_{\G} \circ P$. Therefore $\cP_{\Sigma_{B_\G}}(\phi) = \ln \lambda(B_{\G} \circ P) $. The result follows by Theorem~\ref{thm:erformula}.
	\end{proof}
	
	\noindent In particular when $\mu_P$ is the Parry measure, the two functions $\phi \equiv 0$ and $\phi_P$ as defined above satisfy $\phi_P - \phi = h\circ \sigma - h + c$, for some function $h:\Sigma_A \rightarrow \mathbb{R}$ and a constant $c \in \mathbb{R}$. It is known that two functions with summable variation satisfying the above relation have same equilibrium measure. For ease of computation above, the potential $\phi = 0$ is preferred in the case of the Parry measure.

	\subsection*{Connection of escape rate with perturbation of shifts}
	In~\cite{HA:23}, the authors derived that for SFT $\Sigma_A$ with Parry measure, the escape rate into the Markov hole $H_\G$ is given by the difference of the topological entropy of $\Sigma_A$, which is $\ln \lambda(A)$, and the topological entropy of the new SFT $\Sigma_{B}$ consisting of those sequences in $\Sigma_A$ which do not contain words from $\G$, which is $\ln \lambda(B)$ (see Corollary~\ref{cor:parry}). 
	In a paper by Lind~\cite{Lind_pert}, the perturbation of subshifts is studied. The notion of escape rate into holes is not discussed explicitly. However, the question of estimating the difference in topological entropies (equivalent to `drop in spectral radius') by forbidding a single word $u$ of length $r$ is considered. In the context of our work, this is precisely the escape rate into a single cylinder in the SFT when the underlying measure is the Parry measure. Lind proved that the difference between the topological entropies of $\Sigma_A$ and $\Sigma_B$ lies between positive (computable) constants times $\lambda(A)^{-r}$. The characteristic polynomial of $B$ involves the characteristic polynomial of $A$, a cofactor of the matrix $z\mathbf{I}-A$, and the correlation polynomial of the word $u$ (as in~\eqref{eq:polynomial}, see also Theorem~\ref{thm:relate}). Perturbations of subshifts are studied for one-dimensional subshifts when more than one word is forbidden by Ramsey in~\cite{Ramsey}. Moreover, this issue is addressed by Pavlov~\cite{pavlov} for multidimensional shifts of finite type. If we were to pose this issue in our setting of SFTs with a Markov measure, an appropriate question to ask would be (with notations as in the previous subsection) whether the ratio {\small $\dfrac{\cP_A(\phi_P) - \cP_{B_\G}(\phi_P)}{\mu(H_\G)}$} is bounded between two positive constants for all collections $\G$ consisting of $k$ words each of length $r$, for sufficiently large $r$. This question is addressed by Pollicott and Ferguson~\cite{Ferg_polli} in the context of local escape rate. 

	\section{Escape rate using recurrence relations}\label{sec:rec_relations}
	In this section, we derive certain recurrence relations that give us another method to compute the escape rate into a Markov hole $H_\G$. These relations will feature weighted correlation polynomial functions between the words in $\G$, which we define below. Let $\Sigma_A$ be an irreducible SFT with Markov measure $\mu=\mu_P$. Let $\mathbf{p}$ be the stationary vector of $P$.
	
	\begin{definition}[Weighted correlation polynomial] 
		Consider two words $u=u_1\dots u_\ell$ and $v=v_1\dots v_m$ with symbols from $\Sigma$ and $\ell,m\ge 2$.
		\begin{itemize}
			\item Fix $0\le s\le \ell-1$. We say that $s\in (u,v)$, if $u_{s+1}\dots u_\ell=v_1\dots v_{\ell-s}$.
			\item Set
			\[
			c_{s,u,v} := \begin{cases}
				1, \ \text{if } s\in (u,v),\\
				0, \ \text{otherwise}.
			\end{cases}
			\]
		\end{itemize} 
		The \emph{weighted correlation polynomial} of $u,v$ is defined as  
		\[
		\tau_{u,v}(z):=\begin{cases}
			\sum\limits_{s=0}^{\ell-1} c_{s,u,v}\delta_{s,u,v}z^{m+s-\ell}, & \text{if }\ell< m,\\
			\sum\limits_{s=\ell-m}^{\ell-1} c_{s,u,v}\delta_{s,u,v}z^{m+s-\ell}, & \text{if }\ell\ge m,
		\end{cases}
		\] where $\delta_{s,u,v}=P_{v_{\ell-s}v_{\ell-s+1}}\dots P_{v_{m-1}v_m}$, for $s>0$. By convention, when $\ell \ge m$, $\delta_{\ell-m,u,v}:=1$.
	\end{definition}
	
	\begin{definition}[Prime word]
		A word $u=u_1\dots u_\ell$ is known as a \textit{prime word} if, $s\in (u,u)$ if and only if $s=0$.
	\end{definition}
	
	\begin{remark}
		A few remarks about the weighted correlation polynomial.
		\begin{enumerate}
			\item If $u=v$, we use shorthand $\tau_u(z)$ for $\tau_{u,u}(z)$, and call it the \textit{weighted autocorrelation polynomial} of $u$. Observe that since $c_{0,u,u}=1$, the polynomial $\tau_u(z)$ has a constant term. 
			\item If $\ell-1\in (u,v)$, then $u_\ell=v_1$. We denote $\delta_v:=\delta_{\ell-1,u,v}=\mu(C_v)/p_{v_1}$.  
			\item For a prime word $u$, $\tau_u(z)=1$.
		\end{enumerate}
	\end{remark}
	
	Let $H_{\G}$ be a Markov hole in $\Sigma_A$ given by a finite collection $\G$. In this section, we assume that $\G$ is {\it reduced}; that is, no word of $\G$ appears as a subword of any other word in $\G$. Note that $\G$ can contain words of different lengths, even words of length 1. 
	
	For $i\in\Sigma\setminus \G$, let $\mu(n,i)$ be the total measure of cylinders in $(\Sigma_A,\mu)$ based at words of length $n$ ending with $i$, not containing any word from $\G$ as a subword. Let $F_{i}(z)=\sum_{n\ge 1}\mu(n,i)z^n$ be the generating function of $\mu(n,i)$. Note that $\mu(1,i) = p_i$. 
	
	The following result gives another method to compute the escape rate into $H_\G$ in terms of these generating functions. We define $\F(z):= \sum_{i \in \Sigma\setminus \G} F_i(z)$, which is the generating function of $\sum_{i \in \Sigma\setminus \G} \mu(n,i)$.
	
	\begin{theorem}\label{thm:esc_rate_as_smallest_pole}
		With the notations as above, the radius of convergence of $\F(z)$ is given by $\text{exp}(\rho(H_\G))$. Moreover, $\text{exp}(\rho(H_\G))$ is a pole of $\F(z)$.
	\end{theorem}
	
	\begin{proof}
		Observe that $\sum_{a \in \Sigma\setminus \G}\mu(n,a)$ is the total measure of cylinders in $(\Sigma_A,\mu)$ based at words of length $n$, not containing any word from $\G$ as a subword. Moreover, as argued in the proof of Theorem~\ref{thm:erformula}, $\sum_{a \in \Sigma\setminus \G}\mu(m+r,a)\le \mu(\mathcal{W}_{m})\le \sum_{a \in \Sigma\setminus \G}\mu(m+1,a)$, for all $m$, where $r$ is the length of the longest word in $\G$.  
		Hence the radius of convergence of $\F(z)$ is given by 
		\[
		\lim_{m\to\infty}\mu(\mathcal{W}_{m})^{-1/m} = \text{exp}(\rho(H_\G)),
		\]
		by the definition of $\rho(H_\G)$. Using~\eqref{eq:erate-BoP} and the Jordan canonical form of $B\circ P_{r-1}$, it can be proved that $\F(z)$ diverges at $z=\text{exp}(\rho(H_\G))$. 
	\end{proof}
	
	In what follows, we will obtain a system of linear equations to solve for $F_a(z)$, $a \in \Sigma\setminus \G$. Upon solving the system of linear equations, we will see that $\F(z)$ is a rational function, and hence by Theorem~\ref{thm:esc_rate_as_smallest_pole}, the escape rate into the Markov hole $H_\G$ is the logarithm of its smallest pole in modulus. We introduce some notations that will be used in the system of equations.
	
	For each $u\in\G$, let $\nu(n,u)$ be the total measure of cylinders $(\Sigma_A,\mu)$ based at words of length $n$ not containing any word from $\G$ as a subword except for a single appearance of the word $u$ at the end. Let $G_{u}(z)=\sum_{n\ge 1}\nu(n,u)z^n$ be the associated generating function. Note that $\nu(n,u) = 0$, for all $n < |u|$, where $|u|$ denotes the length of $u$, and $\nu(|u|,u) = \mu(C_u)$. In particular, if $u\in \G$ is the symbol $a$, say, then $\nu(1,u) = p_a$. 
	
	Below, $i(u)$ and $t(u)$ refer, respectively, to the first and last symbol of the word $u$. In general, $t_j(u)$ refers to the terminal subword of length $j$ of a word $u$. Also, $\chi_{a,b}=1$ if $a=b$, and $\chi_{a,b}=0$, otherwise. We have the following result. 
	
	\begin{theorem}\label{thm:rec_relations}
		(With the notations given above) For each $i \in \Sigma\setminus \G$ and $u\in \G \setminus \Sigma$, the generating functions $F_i$ and $G_{u}$, satisfy the following system of linear equations
		\begin{eqnarray}
			F_i(z)-p_i z+  \sum_{v \in \G\setminus \Sigma} \chi_{i,t(v)} \, G_{v}(z) &=& z\sum_{j \in \Sigma\setminus \G} F_j(z) P_{ji}, \label{eq:final0} \\
			z^{|u|-1}\delta_{u} F_{i(u)}(z) &=& \sum_{v\in\G\setminus \Sigma}  \tilde{\tau}_{v,u}(z) \,G_{v}(z), \label{eq:final}
		\end{eqnarray}
		where $\tilde{\tau}_{v,u}(z):= \tau_{v,u}(z) - z^{|u|-1}c_{|v|-1,v,u}\delta_{u}$.
	\end{theorem}
	\begin{proof}
		Let $w=w_1\dots w_n$ be an allowed word of length $n$ in $\Sigma_A$ that does not contain any word from $\G$ as a subword. Attach a symbol $i\in\Sigma\setminus \G$ at the end of $w$. There are two possibilities. Either $w_1\dots w_n i$ does not contain a word from $\G$ as a subword, or $w_1\dots w_ni$ contains a word, say $v\in\G \setminus \Sigma$ as a subword, in which case, it can only appear at the end of $w_1\dots w_ni$ and $t(v)=i$. In the latter case, if $v \in \G \cap \Sigma$, then $v= i$, which contradicts the fact that $i \in \Sigma\setminus \G.$ Hence for each $i\in\Sigma\setminus \G$ and $n\ge 1$,
		\[
		\mu(n+1,i)+\sum_{v\in\G \setminus \Sigma} \chi_{i,t(v)} \, \nu(n+1,v)=\sum_{j\in \Sigma\setminus \G}\mu(n,j)P_{ji}.
		\]
		Take summation over $n\ge 1$ after multiplying with $z^{n+1}$ to get the desired first set of linear equations. Here we use the fact that, for $i \in \Sigma\setminus \G$, $\mu(1,i) =p_i$, and $\nu(1,v) = 0$ for all $v \in \G \setminus \Sigma$.
		
		
		For the second set of equations, fix a word $u \in \G \setminus \Sigma$. Let $u = u_1\dots u_\ell$, with $|u| = \ell\ge 2$. Suppose $w=w_1\dots w_n$ is an allowed word of length $n$ in $\Sigma_A$, which does not contain any word from $\G$ as a subword. Then, the word $wu$, obtained by attaching $u$ to $w$ as a suffix, can be written in either of the following two ways. Either there is a $v \in \G$ with $0< s\in(v,u)$, such that $wu = w'vt_{\ell+s-|v|}(u)$ and $w'v$ has a unique appearance of a word from $\G$ at the end, where $w'=w_1\dots w_{n-s}$ and $t_{\ell+s-|v|}(u)=u_{|v|-s+1}\dots u_\ell$ is the terminal subword of $u$ of length $\ell+s-|v|$. In this case, since $\G$ is reduced, $\ell+s-|v| \ge 1$, for all $s \in (v,u)$, and hence $t_{\ell+s-|v|}(u)$ is well defined. Or else, if no such $v$ exists, then we take $v=u$, which is the only appearance of a word from $\G$ at the end of $wu$, and in this case, $s=0\in (v,u)$. Moreover, in both cases, $v \in \G \setminus \Sigma$ due to the fact that $w$ does not contain a word from $\G$ as a subword and also that $\G$ is reduced.
		
		\noindent Therefore, for each $u\in \G \setminus \Sigma$ and $n\ge 0$, we have the following:
		\begin{eqnarray*}
			\mu(n+1,i(u))\delta_{u} &=& \sum_{v\in\G \setminus \Sigma}  \ \sum_{\substack{0 \, \le \, s \, < \, |v|-1, \\ s\in (v,u)}} \nu(n+|v|-s,{v}) \, \delta_{s,v,u}\\
			&=& \sum_{v\in\G \setminus \Sigma}\left(\sum_{s=0}^{|v|-1} c_{s,v,u}\nu(n+|v|-s,v)\delta_{s,v,u} - c_{|v|-1,v,u}\nu(n+1,v)\delta_{u}\right)\\
			&=& \sum_{v\in\G \setminus \Sigma}\left(\sum_{s=0}^{|v|-1} c_{s,v,u}\nu(n+|v|-s,v)\delta_{s,u} - \chi_{t(v),i(u)}\nu(n+1,v)\delta_{u}\right).
		\end{eqnarray*}
		Note that $\nu(n,v)=0$, for all $0\le n\le |v|-1$. For each $u \in \G \setminus \Sigma$, take summation over $n\ge 0$ after multiplying with $z^{n+\ell}$ to get the desired second set of linear equations. 
	\end{proof}

	\begin{remark}\label{rem:esc_rate_as_smallest_pole}
		Clearly, the generating function $\F(z)=\sum_{a\in\Sigma\setminus \G}F_a(z)$ is a rational function. Further, by Theorems~\ref{thm:erformula} and~\ref{thm:esc_rate_as_smallest_pole}, $\lambda(B_\G\circ P_{r-1})^{-1}$ is its radius of convergence. 
		We can compute the rational function $\F(z)$ by solving the system of equations given 
		in~\eqref{eq:final0} and~\eqref{eq:final}. This system has $|\Sigma\setminus \G|+|\G \setminus \Sigma|$ many linear equations in $F_i(z), G_{u}(z)$, for $i \in \Sigma\setminus \G$ and $u \in \G \setminus \Sigma$. To compute the escape rate, we find the smallest real positive pole of $\F(z)$.
		
		\noindent By arguments as in the proof of the preceding result, we get that the generating functions $G_k(z)$, for $k\in \G\cap \Sigma$ satisfy
		\[
		p_k z + G_k(z) = z\sum_{j \in \Sigma\setminus \G} F_j(z) P_{ji}.
		\]
		However these generating functions play no role in the computation of $\F(z)$. 
	\end{remark}
	
	\subsection{Matrix form of the system of equations~\eqref{eq:final0}, ~\eqref{eq:final}}\label{sec:mat_form}
	
	Let $\Sigma=\{1,\dots,N\}$ and $\G$ be a given finite collection of allowed words in $\Sigma_A$. Denote $\G_1=\{k_1, \dots, k_m\}$, $\G \setminus \G_1=\{u_1,\dots,u_q\}$ for some $m,\,q \ge 0$ and $\Sigma\setminus \G= \Sigma \setminus \G_1$ as before. We fix the order on each set $\Sigma\setminus \G, \, \G_1$ and $\G \setminus \G_1$. 
	
	Throughout the paper, we denote the identity matrix by $\mathbf{I}$; the size will be clear from the context. We will now write the system of equations~\eqref{eq:final0} and~\eqref{eq:final} in matrix form. The matrix form will allow us to make certain interesting observations. Recall that $\chi_{a,b}$ equals 1 if $a=b$ and 0 otherwise. The system of equations~\eqref{eq:final} can be written as 
	\begin{eqnarray}\label{eq:matrixform}
		\mathbf{M}(z) \begin{pmatrix}
			F(z) \\ G(z)
		\end{pmatrix} &=& \begin{pmatrix}
			z \mathbf{p}\vert_{\Sigma\setminus \G} \\ \mathbf{0}
		\end{pmatrix},
	\end{eqnarray}
	where $F(z)=(F_i(z))_{i \in \Sigma\setminus \G}^T$, $G(z)=(G_{u}(z))_{u\in\G\setminus\Sigma}^T$ both are column vectors, $\mathbf{p}\vert_{\Sigma\setminus \G}$ is the vector obtained from the stationary vector $\mathbf{p}$ of $P$ by removing the rows corresponding to symbols in $\G$, $\mathbf{0}$ is zero vector (column) of appropriate size, and 
	\begin{eqnarray*}
		\mathbf{M}(z)  =     \left( 
		\begin{array}{c|c} 
			\mathbf{P}(z) & \mathbf{T}(z) \\ 
			\hline 
			\mathbf{E}(z) & \mathbf{C}(z) 
		\end{array} 
		\right),
	\end{eqnarray*}
	with 
	\begin{eqnarray}
		\mathbf{P}(z) &=& \mathbf{I}-zP^T\vert_{\Sigma\setminus \G},\nonumber\\
		\mathbf{T}(z) &=& \left(\chi_{i,t(v)}\right)_{i \in \Sigma\setminus \G, \, v\in\G\setminus \Sigma}, \nonumber\\
		\mathbf{E}(z) &=& \left(z^{|u|-1}\delta_{u}\chi_{j,i(u)}\right)_{u\in\G\setminus\Sigma, \, j \in \Sigma\setminus \G}, \nonumber\\
		\mathbf{C}(z) &=& -\left(\tilde{\tau}_{u,v}\right)_{u,v\in \G\setminus\Sigma}, \label{eq:wcorrmat}
	\end{eqnarray}
	where $P\vert_{\Sigma\setminus \G}$ denotes the principal submatrix of $P$ obtained by removing all the rows and columns corresponding to the symbols in $\G$. Observe that the matrix $\mathbf{T}$ is a $0-1$ matrix independent of $z$ and records the terminal symbols of words in $\G \setminus \Sigma$. Moreover, it has a single 1 entry in each column. Further, the matrix $\mathbf{E}(z)$ has a single non-zero entry in each row. The matrix $\mathbf{C}(z)$ records the weighted correlation polynomials of words in $\G \setminus \Sigma$. Using the matrix form and inverse of a block matrix, we obtain
	\[
	\F(z) = z\sum \mathbf{P}(z)^{-1} \left[\mathbf{I} 
	+ \mathbf{T}(z) \left( \mathbf{C}(z)-\mathbf{E}(z)\mathbf{P}(z)^{-1}\mathbf{T}(z)) \right)^{-1} \mathbf{E}(z)\mathbf{P}(z)^{-1}\right] \mathbf{p}\vert_{\Sigma\setminus \G},
	\]
	provided $\mathbf{P}(z)$ and $\mathbf{C}(z)-\mathbf{E}(z)\mathbf{P}(z)^{-1}\mathbf{T}(z)$ are invertible. By Theorems~\ref{thm:location} and~\ref{thm:esc_rate_as_smallest_pole}, $\text{exp}(\rho(H_\G)$ is a root of $\text{det}(\text{det}(\textbf{P}(z))\textbf{C}(z)-\textbf{E}(z)\text{adj}(\textbf{P}(z))\textbf{T}(z))$. This is illustrated in Theorem~\ref{thm:erone} for a single cylinder as the hole, in which case, it turns out that $\text{exp}(\rho(H_\G))$ is its smallest positive real root. Moreover, in this case, it is exactly the determinant of the matrix $\mathbf{I}-zB_\G\circ P_{r-1}$ (see Theorem~\ref{thm:relate}). 
	\medskip

	\noindent The following example illustrates the results we have discussed so far.
	
	\begin{exam}\label{exam:compnew}
		Consider the SFT $\Sigma_A$ with Markov measure $\mu_P$ and a Markov hole $H_\G$ where
		\[
		\Sigma=\{1,2,3\}, \ A=\begin{pmatrix}
			1&1&1\\1&1&0\\1&1&1
		\end{pmatrix}, \ P=\begin{pmatrix}
			1/5&2/5&2/5\\
			9/10&1/10&0\\
			1/10&1/10&4/5
		\end{pmatrix}, \ \G=\{12,222,3\} \]
		and $\mathbf{p} = (3/11, 2/11,6/11)^T$ is the stationary vector of $P$.


			\begin{itemize}
				\item Calculation of escape rate into $H_\G$ using Theorem~\ref{thm:esc_rate_as_smallest_pole}: \\
				Let $u=12$ and $v=222$, the elements of $\G\setminus\Sigma$. Simple computations give $\tilde{\tau}_{u,v}(z) = \tilde{\tau}_{v,u}(z) =0$, $\tilde{\tau}_{u,u}(z)=1$, $\tilde{\tau}_{v,v}(z)=1+z/10$. Hence the system of linear equations~\eqref{eq:final0},~\eqref{eq:final} can be written as 
				\[
				\left( 
				\begin{array}{c|c} 
					\begin{matrix}
						1-\frac{z}{5} & -\frac{9z}{10}\\
						-\frac{2z}{5} & 1-\frac{z}{10}
					\end{matrix} & \begin{matrix}
						0 \  \ \ \ & 0\\
						1 \ \ \  \ & 1
					\end{matrix} \\ 
					\hline 
					\begin{matrix}
						\frac{9z}{10} \ \ & 0  \\
						0 \ \ & \frac{2z^2}{100} 
					\end{matrix} & \begin{matrix}
						-1 & 0 \\
						0 & -1-\frac{z}{10}
					\end{matrix}
				\end{array}
				\right)
				\begin{pmatrix}
					F_1(z)\\F_2(z)\\G_{u}(z)\\G_{v}(z)
				\end{pmatrix}=\begin{pmatrix}
					\frac{3z}{11}\\ \frac{2z}{11}\\0\\0
				\end{pmatrix}
				\]
				Solving these, we get \[
				\F(z) = F_1(z)+F_2(z)=\dfrac{z(7z^2+80z+250)}{110(z-5)}. 
				\] The smallest real pole of $\F(z)$ is 5, hence $\rho(H_\G) = \ln (5)$. 
				\item  Calculation of escape rate into $H_\G$ using Theorem~\ref{thm:erformula}:\\
				Since $r = 3$, we need to compute the spectral radius of $B_{\G} \circ P_2$. Since all the rows and columns having labels containing a word of $\G$ are 0 in $\B_\G$, the spectral radius of $B_{\G} \circ P_2$ is the same as the spectral radius of its submatrix only with rows and columns having labels $11,21$ and $22$. The only non-zero entries of this submatrix are,
				\[
				(B_{\G} \circ P_2)_{(11)(11)} = (B_{\G} \circ P_2)_{(21)(11)} = 1/5, \ \text{ and} \ (B_{\G} \circ P_2)_{(22)(21)} = 9/10. 
				\]
				Hence $\rho(H_\G)=-\ln(\lambda(B_\G\circ P_2))=\ln(5)$.
				Thus this method turns out to be much simpler than using the system of equations. 
			\end{itemize}
		\end{exam}


			\subsection{Escape rate into a cylinder}
			In this section, we restrict ourselves to the case where the collection $\G$ contains only one word, say $u = u_1 \dots u_r$, $r \ge 2$. Let $\textbf{e}_i$ denote the column vector having its $i^{th}$ entry $1$ and the rest of its entries $0$. The size of $\textbf{e}_i$ will be clear from the context. 
			
			Before we proceed to one of our main results for another method to compute the escape rate into a single cylinder, we recall formulas (see~\cite{Higham}) for the inverse and the determinant of the matrix obtained by performing a rank one perturbation to an invertible matrix.
			
			\begin{lemma}\label{lemma:det_inv_sum}
				Let $A$ and $B = \alpha \mathbf{e}_i\mathbf{e}_j^T$ be two matrices of size $n \ge 2$, where $\mathbf{e}_i$ and $\mathbf{e}_j$ are column vectors of size $n$ as defined earlier and $\alpha$ is a scalar. Note that $B$ has rank 1.
				\begin{itemize}
					\item Determinant of $A+B$ is given by $\text{det}(A+B)=\text{det}(A)+\alpha \mathbf{e}_j^T\text{adj}(A)\mathbf{e}_i$.
					\item Sherman-Morrison formula for inverse: If $A$ and $A+B$ are invertible, then the inverse of $A+B$ is given by
					\[
					(A+B)^{-1}=A^{-1} - \frac{\alpha}{1+\alpha \mathbf{e}_j^T A^{-1} \mathbf{e}_i} A^{-1} \mathbf{e}_i \mathbf{e}_j^T A^{-1} .
					\]
				\end{itemize}
				Note that $\mathbf{e}_i \mathbf{e}_j^T$ is a square matrix with all its entries 0 except for its $ij^{th}$ entry, which is 1, and $\mathbf{e}_j^T\text{adj}(A)\mathbf{e}_i$ gives the $ji^{th}$ entry of $\text{adj}(A)$, the adjugate matrix of $A$.
			\end{lemma}
			
			We need the following lemma, which follows from the basic facts of linear algebra.
			
			\begin{lemma}\label{lemma:adj}
				Let $P$ be an irreducible stochastic matrix of size $N$ with stationary vector $\mathbf{p}$. Then the adjugate matrix of $\mathbf{I}-P$ is given by $\text{adj}(\mathbf{I}-P)=c\mathbf{1}\textbf{p}^T$, where $c=\lim_{z\rightarrow 1}(1-z)^{-1}\text{det}(\mathbf{I}-zP)$ which is positive.
			\end{lemma} 
			
			\begin{theorem}\label{thm:erone}
				Let $u=u_1\dots u_r$ be an allowed word in $\Sigma_A$ of length $r \ge 2$. Let $z_0$ be the smallest real root of the polynomial
				\begin{equation}\label{eq:polynomial}
					f_u(z):=\tilde{\tau}_{u}(z)\text{det}(\mathbf{I}-zP)+ z^{r-1}\delta_{u}\mathbf{e}_{u_r}^T\text{adj}(\mathbf{I}-zP)\mathbf{e}_{u_1},
				\end{equation}
				where $\tilde{\tau}_{u}(z)=\tau_{u}(z)- z^{r-1} c_{r-1,u,u}\delta_u$. Then, the escape rate $\rho(C_u)$ into the cylinder based at $u$ with respect to the Markov measure $\mu_P$ is given by $\ln z_0$.
			\end{theorem}
			
			\begin{proof}
				Let $\textbf{p}$ be the stationary vector of $P$.
				The system of equations~\eqref{eq:matrixform} is given by
				\begin{eqnarray*}
					\left( 
					\begin{array}{c|c} 
						\mathbf{I}-zP^T& \mathbf{e}_{u_r} \\ 
						\hline 
						z^{r-1}\delta_u\mathbf{e}_{u_1}^T & -\tilde{\tau}_u(z) 
					\end{array} 
					\right)\begin{pmatrix}
						F(z) \\ G_u(z) 
					\end{pmatrix} = \begin{pmatrix}
						z\textbf{p}\\ 0
					\end{pmatrix},
				\end{eqnarray*} where $F(z)=(F_1(z)\dots F_N(z))^T$.
				Hence, we can compute that 
				\begin{eqnarray}\label{eq:Fproof}
					F(z)&=&z\tilde{\tau}_u(z)L(z)^{-1}\textbf{p},
				\end{eqnarray}
				where $L(z)=\tilde{\tau}_u(z)(\mathbf{I}-zP^T)+z^{r-1}\delta_u\mathbf{e}_{u_r}\mathbf{e}_{u_1}^T$. Observe that by Lemma~\ref{lemma:det_inv_sum}, $\text{det}(L(z))=\tilde{\tau}_u(z)^{N-1}f_u(z)$, which is a nonzero polynomial in $z$ (for instance, at $z=1$, it equals $\tilde{\tau}_u(1)^{N-1}f_u(1)=\tilde{\tau}_u(1)^{N-1}\mu(u)\ne 0$). 
				
				If $\mathbf{I}-zP$ is invertible, using the Sherman-Morrison formula for $L(z)$ (from Lemma~\ref{lemma:det_inv_sum}), we get 
				\begin{eqnarray}\label{eq:Llinear}
					L(z)^{-1} = \left((\mathbf{I}-zP^T)^{-1} - \frac{z^{r-1} \delta_u D(z)}{f_u(z)} (\mathbf{I}-zP^T)^{-1} \mathbf{e}_{u_r}\mathbf{e}_{u_1}^T(\mathbf{I}-zP^T)^{-1}\right), \nonumber
				\end{eqnarray}
				where $D(z)=\text{det}(\mathbf{I}-zP^T)$.
				
				Post-multiplying and pre-multiplying $(\mathbf{I}-zP)^{-1}$ in $\textbf{p}^T = \textbf{p}^T P$ and $P\mathbf{1}=\mathbf{1}$, and substituting $P = \frac{1}{z}(zP-\mathbf{I}+\mathbf{I})$, we get
				\[
				\sum (\mathbf{I}-zP^T)^{-1}\textbf{p} = \dfrac{1}{1-z}, \ \ \sum (\mathbf{I}-zP^T)^{-1} \mathbf{e}_{u_r}\mathbf{e}_{u_1}^T(\mathbf{I}-zP^T)^{-1}\textbf{p}= \dfrac{p_{u_1}}{(1-z)^2}.
				\]
				Substituting these in~\eqref{eq:Fproof}, we get \begin{eqnarray}\label{eq:sumrat}
					\F(z) =	\sum_{a\in\Sigma} F_a(z) 
					&=& z\tilde{\tau}_u(z) \left( \dfrac{1}{1-z} - \dfrac{z^{r-1} \delta_u D(z)}{f_u(z)}\dfrac{ p_{u_1}}{(1-z)^2} \right)\nonumber\\
					&=& \dfrac{z\tilde{\tau}_u(z)}{(1-z)f_u(z)} \left( f_u(z)-z^{r-1}p_{u_1}\delta_u\dfrac{D(z)}{1-z} \right).
				\end{eqnarray}
				Observe that by Lemma~\ref{lemma:adj} we have,
				\[
				\lim_{z\rightarrow 1} f_u(z)-z^{r-1}p_{u_1}\delta_u\dfrac{D(z)}{1-z} = f_u(1)-p_{u_1}\delta_u \lim_{z\rightarrow 1}\dfrac{D(z)}{1-z} = 0,
				\]
				This implies that $z=1$ is not a pole for $\F(z)$, which is indeed the case, by Theorem~\ref{thm:esc_rate_as_smallest_pole} since $\text{exp}(\rho(C_u))$ is the smallest pole of $\F(z)$, which is real and greater than 1. 
				By the form of $\F(z)$ as a rational function, $\text{exp}(\rho(C_u))$ is a real root of $f_u(z)$. Hence, $\text{exp}(\rho(C_u))\ge z_0$ which is the smallest real root of $f_u$. We need to prove that they are equal. Clearly, $z_0>1$, as otherwise $\F(z)$ will have its smallest pole less than one which is a contradiction. We consider two cases now. 
				\begin{itemize}
					\item If 1 is the only real positive eigenvalue of $P$, 
					we have 
					\[
					f_u(z_0)-z_0^{r-1}p_{u_1}\delta_u \dfrac{D(z_0)}{1-z_0}\ne 0,
					\]
					hence $z_0$ is a pole of $\F(z)$. Thus $\text{exp}(\rho(C_u))\le z_0$.
					\item If $\theta$ is the second largest positive real eigenvalue of $P$ other than 1, recalling Theorem~\ref{thm:location}, we know that $\exp(\rho(C_u))\in (1,\theta^{-1}]$, and hence $z_0\in(1,\theta^{-1}]$. \\
					If $z_0=\theta^{-1}$, then $\text{exp}(\rho(C_u))=\theta^{-1}$. On the other hand, if $z_0\in (1,\theta^{-1})$, then 
					\[
					f_u(z_0)-z_0^{r-1}p_{u_1}\delta_u \dfrac{D(z_0)}{1-z_0}\ne 0,
					\]
					hence $z_0$ is a pole of $\F(z)$. Thus $\text{exp}(\rho(C_u))\le z_0$.
				\end{itemize}
			\end{proof}
			

			\begin{remark}\label{rem:compareboth}
				If the word $u$ has length 1, the spectral radius of $B_\G\circ P$ is just the spectral radius of the principal minor of $P$ obtained by removing the row and column of $P$ with label $u$. Also, if the word $u$ has length 2, by Lemma~\ref{lemma:det_inv_sum}, $\text{det}(\mathbf{I}-zB_\G\circ P)=f_{u}(z)$. Hence Theorem~\ref{thm:erone} is particularly useful when $r\ge 3$. When $\G$ contains a single word $u$ of length $r>2$, the identity $\text{det}(\mathbf{I}-zB_\G\circ P_{r-1})=f_{u}(z)$ still holds. This requires a series of lemmas. We state the main theorem here, and the details are given in Section~\ref{sec:relate}.
		\end{remark}
		
		\begin{theorem}\label{thm:relate}
			For $u=u_1\dots u_{r}$ and $\G=\{u\}$, 
			\[
			\text{det}(\mathbf{I}-zB_\G\circ P_{r-1})=f_u(z).
			\]
		\end{theorem}
		
		\begin{remark}
			At this juncture, we point out an interesting relation of the polynomial $f_u$ with the zeta function. The zeta function of a map summarizes information about the periodic points of the map. A generalized version of the zeta function is defined by assigning ``weights'' to periodic orbits. Recall the potential function $\phi = \phi_P$ for which $\mu_P$ is the equilibrium state, defined in Section~\ref{subsec:er_and_pressure}. The $\phi|_{B_{\G}}$-weighted zeta function for $\Sigma_{B_{\G}}$ is defined as
			\[
			\zeta(z, \phi|_{B_{\G}}) = \text{exp} \left\lbrace \sum\limits_{n=1}^{\infty} \frac{z^{n}}{n}  \left( \sum\limits_{  x \in \Sigma_{B_{\G}} \ :\ \sigma^n x = x  } e ^{\phi^n(x)}\right) \right\rbrace,
			\]
			where $\phi^n(x) =\sum\limits_{i=0}^{n-1} \phi \circ \sigma^{i}_{B_{\G}}(x)$ denotes the ergodic sum. It is a known result that $\zeta(z, \phi|_{B_{\G}}) = 1/\text{det}(\mathbf{I}-zB_\G\circ P_{r-1})$, see \cite{Parry_pollicott_ast}. 
			Hence, when $\G$ consists of a single word, the generating function $\F(z)$ is product of some polynomial and the weighted zeta function $\zeta(z, \phi|_{B_{\G}})$ (see~\eqref{eq:sumrat}). In general, when $|\G| = k > 1$, following the same argument as given in proof of Theorem~\ref{thm:erone}, we get $F(z) = z \tilde{\tau}_u(z) L(z)^{-1}p$, where $L(z)$ is a rank $k$ perturbation of the matrix $\tilde{\tau}_u(z)(\mathbf{I}-zP^T)$. Hence by employing the Woodbury formula for the inverse of a rank $k$ perturbed matrix, we expect to obtain the escape rate into $H_{\G}$ in terms of the smallest real root of $\text{det}(\mathbf{I}-zB_\G\circ P_{r-1})$. Hence, we believe that in the general case as well, $\F(z)$ can be expressed as the product of a rational function and the weighted zeta function $\zeta(z, \phi|_{B_{\G}})$. For instance, consider Example~\ref{exam:varyr} in which $\G$ consists of three words. Simple calculations show that, for $r=2,3$, $\F(z)$ is in fact some rational function times the zeta function $\zeta(z, \phi|_{B_{\G}})$. 
		\end{remark}
		
		The following result states that the escape rate into cylinders based at sufficiently long words can be made arbitrarily small.
		
		\begin{theorem}\label{thm:large_r}
			Given any $\epsilon>0$, there exists $R\ge 1$ such that for all words $u$ of length $r\ge R$, $\rho(C_u)<\epsilon$. 
		\end{theorem}
		\begin{proof}
			By Lemma~\ref{lemma:adj}, each entry of $\text{adj}(\mathbf{I}-P)$ is positive. Let $z_1>1$ be such that each entry of $\text{adj}(\mathbf{I}-zP)$ remains positive in the entire interval $[1,z_1]$ and let $M_1=\max_{i,j\in\Sigma}\text{adj}(\mathbf{I}-z_1P)_{ij}$. Let $\epsilon>0$ be given. 
			
			Assume that $z_1<\min\{e^\epsilon,\theta^{-1},P_{max}\}$, where $\theta<1$ is a real positive eigenvalue of $P$ (if exists) and $P_{max}=\max_{i,j\in\Sigma}P_{ij}$. We now prove that there exists $R\ge 1$ such that for each word $u$ of length $r\ge R$, $f_u(z)$ has a root in the interval $[1,z_1]$ (where $f_u(z)$ is defined in~\eqref{eq:polynomial}). 
			
			With the above assumptions, the smallest real root of $f_u(z)$ lies in the interval $(1,z_1]$, and hence $\rho(C_u)\le \ln z_1<\epsilon$, by Theorem~\ref{thm:erone}. Since $z_1P_{max}<1$, choose $R\ge 1$ such that
			\[
			(z_1P_{max})^{R-1}<-\dfrac{\text{det}(\mathbf{I}-z_1P)}{M_1}.
			\]
			Let $r\ge R$ and $u$ be an allowed word of length $r$ beginning with $a$ and ending with $b$. Since $\tilde{\tau}_{u}(z_1)>1$, $\text{det}(\mathbf{I}-z_1P)<0$ and $\text{adj}(\mathbf{I}-z_1P)_{ba}>0$, we have
			\begin{eqnarray*}
				f_u(z_1) &=& \tilde{\tau}_{u}(z_1)\text{det}(\mathbf{I}-z_1P)+ z_1^{r-1}\delta_{u}\text{adj}(\mathbf{I}-z_1P)_{ba}\\
				&\le & \text{det}(\mathbf{I}-z_1P)+ (z_1P_{max})^{r-1}M_1<0.
			\end{eqnarray*}
			Observing that $f_u(1)>0$, the result follows.
			
		\end{proof}
		
		The above result can be used to construct a large hole with an arbitrarily small escape rate, as in the following corollary.
		
		\begin{corollary}
			Given any $\epsilon>0$, there exists a Markov hole $H$ with $\mu(H)>1-\epsilon$ and $\rho(H)<\epsilon.$
		\end{corollary}
		\begin{proof}
			By Theorem~\ref{thm:large_r}, let $u$ be a word with $\rho(C_u)<\epsilon$. Since $\Sigma_A$ is ergodic with respect to $\mu_P$, there exists $n\ge 1$ so that $\mu_P\left(\bigcup_{i=0}^n \sigma^{-i}C_u\right)>1-\epsilon$. Take $H=\bigcup_{i=0}^n \sigma^{-i}C_u$. The result follows since $\rho(C_u)=\rho(H)$.
		\end{proof}
		\subsection{Special cases}
		In~\cite{BCL}, Bonanno~\textit{et al.} study the issue of computing the escape rate into a single cylinder in the following two cases: (i) when $\mu_P$ is a product measure, and (ii) when $\Sigma$ has only two symbols and $\mu_P$ is any Markov measure. These cases follow from our results. Let us consider these cases one by one. 
		
		\begin{enumerate}
			\item[(i)] Let $\mu_P$ be a product measure. Then $P=\mathbf{1}\textbf{p}^T$ for some probability vector $\textbf{p}$. Using Lemma~\ref{lemma:det_inv_sum}, we can easily compute $\text{det}(\mathbf{I}-zP)=1-z$ and $\text{adj}(\mathbf{I}-zP)=(1-z)I+zP$. Let $u$ be an allowed word of length $r$. The function $f_u(z)$ defined as in~\eqref{eq:polynomial} takes the form $f_u(z)=\tilde{\tau}_{u}(z)(1-z)+z^{r}\mu(u)$. In this setting, Theorem~\ref{thm:esc_rate_as_smallest_pole} is precisely the statement of~\cite[Proposition 2.4]{BCL}. 
			\item[(ii)] Consider the subshift $\Sigma_A$ with only two symbols with the Markov measure $\mu_P$ where $P=(P_{ij})_{1\le i,j\le 2}$. It is easy to see that 
				\[
				f_u(z)=(1-z)(1-(T-1) z)\tilde{\tau}_u(z)+\tilde{\delta}_u z^r +\chi_{u_1u_r}\delta_u(1-Tz)z^{r-1},
				\] 
				where $\tilde{\delta}_u=\delta_u P_{u_ru_1}$ and $T=\text{trace}(P)=P_{11}+P_{22}$. In this setting, Theorem~\ref{thm:esc_rate_as_smallest_pole} is precisely the statement of~\cite[Proposition 4.2]{BCL}.
			\end{enumerate}

			\section{Comparison of the two approaches to compute the escape rate}\label{sec:advantages}
			In the preceding two sections, we have presented two different methods to compute the escape rate into a Markov hole $H_{\G}$ in SFT $\Sigma_A$. The first one as given in Theorem~\ref{thm:erformula}, is a direct formula given by $\rho(H_\G) = -\ln \lambda(B_\G\circ P_{r-1})$, which depends on the higher block presentation $P_{r-1}$ of the stochastic matrix $P$ (recall: $r$ is the length of the longest word in $\G$) and the adjacency matrix $B_{\G}$ of the open system.
			The second method, based on Theorem~\ref{thm:esc_rate_as_smallest_pole}, states that the escape rate is the logarithm of the radius of convergence of certain rational functions. Both methods are useful in their own right; one is more applicable than the other depending on the situation. In this section, we highlight some of the advantages of using either of the methods over the other. 
			
			In the first formula, the adjacency matrix $B_{\G}$ and the higher-order representation $P_{r-1}$ of the stochastic matrix $P$, both are indexed by $\mathcal{L}_{r-1}$ (which is the set of allowed words of length $r-1$ in $\Sigma_A$). The sizes of these matrices grow exponentially in $r$; hence, the method is computationally expensive for large values of $r$. However, for each fixed $r$, the formula does not depend on the number of words in the collection $\G$ determining the hole. On the other hand, in the second method, the size of the matrix $\mathbf{C}(z)$ (see~\eqref{eq:wcorrmat}) grows with the number of elements in $\G$, where $|\G|^2$ many weighted correlation polynomials need to be computed. Hence, for small $r$ and large $|\G|$, the first method may be more effective. We must point out that the words in $\G$ may have little overlapping, and thus, for large $|\G|$, the matrix $\mathbf{C}(z)$ may still be easier to compute. Thus, the second method may be computationally less expensive than the first method in certain situations.
			
			In the second approach, to obtain the rational function (the logarithm of its pole is the escape rate), one is required to derive the system of linear equations consisting of as many equations in total as $|\G\setminus \Sigma| + |\Sigma \setminus \G| $. Hence, the larger the collection $\G \setminus \Sigma$, the more cumbersome it is to form the required system of equations. However, if the collection $\G$ is small enough, then this approach can prove to be more useful.
			
			In Example~\ref{exam:compnew} discussed earlier, computation of escape rate using the first method is much easier than the second. The following example demonstrates a situation where the method of generating functions is better suited.

			\begin{exam}\label{exam:varyr}
				Consider the full shift on two symbols with a product measure $\mu_P$ and a Markov hole $H_\G$, with $\G$ consisting of three words each of length $r$: $u_1=1 1\dots 12$, $u_2=2 2\dots 21$, $u_3=122\dots 221$. We wish to compute the escape rate into the hole $H_\G$. In order to apply Theorem~\ref{thm:erformula}, we need to compute the matrix $P_{r-1}$ whose size is $2^{r-1}$, which grows exponentially with $r$. However, no matter how large $r$ is, applying Theorem~\ref{thm:rec_relations} is much simpler. We need to solve a system of 5 linear equations to be able to arrive at a rational function form of $F_1(z)+F_2(z)$, whose smallest pole is $\text{exp}(\rho(H_\G))$, by Theorem~\ref{thm:esc_rate_as_smallest_pole}. Let us now work out the system of equations in the matrix form~\eqref{eq:matrixform}, as described in Section~\ref{sec:mat_form}. We have
				\[
				M(z)=\begin{pmatrix}
					1-zP_{11} & -zP_{21} & 0 & 1 & 1\\
					-zP_{12} & 1-zP_{22} & 1 & 0 & 0 \\
					z^{r-1}\delta_{u_1} & 0 & -1 & 0 & 1\\
					0 &  z^{r-1}\delta_{u_2} & 0 & -1 & 0\\
					z^{r-1}\delta_{u_3} & 0 & -z^{r-2}\delta_{u_3}/P_{12} & 0 & -1
				\end{pmatrix}.
				\]
				Since $\mu_P$ is a product measure, $P=\mathbf{1}\mathbf{p}$, where $\mathbf{p}^T=(p_1,p_2)$ is a probability vector. For simplicity of analysis, let $p_1=p_2=1/2$ (uniform measure). Then
				\[
				\F(z) = F_1(z)+F_2(z)=\dfrac{-4^rz^4+2^rz^{3+r}-4z^{1+2r}}{4^rz^3-4^rz^4+8z^{2r}+2^{1+r}z^{2+r}+2^rz^{3+r}-8z^{1+2r}}.
				\]
				Table~\ref{table:valuesr} shows the smallest pole of the above rational function (which is nothing but $\text{exp}(\rho(H_\G))$), as $r$ varies between 2 to 8.
			\end{exam}
			\begin{table}[h!]
				\centering
				\begin{tabular}{|c|c|c|c|c|c|c|c|}
					\hline
					$r$ & 2 & 3 & 4 & 5 & 6 & 7 & 8 \\
					\hline
					$\text{exp}(\rho(H_\G))$& 2&2&2&1.21&1.07&1.03&1.01 \\
					\hline
				\end{tabular}
				\caption{Change in escape rate with respect to the length of the words}
				\label{table:valuesr}
			\end{table}
			In particular, when $\G$ consists of a single word $u$ of any length $r$, then the escape rate into the cylinder is given by logarithm of the smallest real root of the polynomial $f_u$ given in~\eqref{eq:polynomial}. Note that this polynomial depends on the weighted correlation polynomial of $u$, the matrix $P$, the measure of the cylinder $C_u$, and the first and last symbol of $u$. By exploiting this fact, we can compare escape rates into two words by fixing some of these parameters and varying the others. For instance, below, we prove two more results about the relationship between escape rates into two cylinders corresponding to words satisfying certain relations. It is not so direct to derive them from the first method, except that when the hole corresponds to a single word, the two methods are related through Theorem~\ref{thm:relate}. Theorem~\ref{thm:compare-prime} below states that in certain classes of words satisfying specific parameters, the prime word in that class achieves the maximum escape rate. 
			
			\begin{theorem}\label{thm:compare-prime}
				Let $u$ be an allowed prime word of length $r\ge 2$ in $\Sigma_A$. If $w\in \mathcal{L}_r$ satisfies $i(w)=i(u)$, $t(w)=t(u)$, and $\delta_w=\delta_u$, then $\rho(C_w)\le \rho(C_u)$. Moreover, if 1 is the only positive real eigenvalue of $P$, equality holds if and only if $w$ is a prime word.
			\end{theorem}
			
			\begin{proof}
				By Theorem~\ref{thm:erone}, observe that  \[
				f_w(z)=(\tilde{\tau}_{w,w}(z)-1)\text{det}(\mathbf{I}-zP)+f_u(z).
				\]
				If there exists a second largest positive real eigenvalue $\theta$ of $P$, then for all $1<z<\theta^{-1}$, we have $\text{det}(\mathbf{I}-zP)<0$. Moreover since $\tilde{\tau}_{w,w}(z)\ge 1$, for real $z>1$, we have $f_w(z)\le f_u(z)$, for all $z\in (1,\theta^{-1})$ and $f_w(z)=f_u(z)$ at $z=1,\theta^{-1}$. By Theorem~\ref{thm:location}, $\rho(C_w)\le \rho(C_u)$. 
				
				Further, if $P$ has no positive real eigenvalue other than 1, then $\text{det}(\mathbf{I}-zP)<0$, for all $z>1$. Hence $f_w(z)\le f_u(z)$, for all $z>1$. Clearly $\rho(C_w)\le \rho(C_u)$.
				
				Now, suppose $P$ has 1 as the only positive real eigenvalue. Observe that if $w$ is not a prime word, then $\tilde{\tau}_{w,w}(z)> 1$, for all $z>1$. Hence $\rho(C_w)= \rho(C_u)$ if and only if $w$ is a prime word.
			\end{proof}
			
			Similarly, the following result states that under certain conditions, the words having the least correlation achieve the maximal escape rate. We skip its proof since it follows along the lines of that for Theorem~\ref{thm:compare-prime}.
			\begin{theorem}
				Let $u$ be an allowed word of length $r\ge 2$ in $\Sigma_A$ with $\tau_{u,u}(z)=1+z^{r-1}\delta_u$. If $w\in \mathcal{L}_r$ satisfies $i(w)=i(u)=t(u)=t(w)$, and $\delta_w=\delta_u$, then $\rho(C_w)\le \rho(C_u)$.
			\end{theorem}

			\section{Proof of Theorem~\ref{thm:relate}}\label{sec:relate}
			
			In this section, we present a proof of Theorem~\ref{thm:relate}. The proof uses a series of lemmas, which we state and prove first.
			
			\begin{lemma}\label{lemma:det-rel}
				For each $r\ge 1$, \[
				\text{det}(\mathbf{I}-zP_{r})=\text{det}(\mathbf{I}-zP).
				\]
			\end{lemma}
			
			\begin{proof}
				The characteristic polynomial of $P$ can be written as 
				\[
				\text{det}(z\mathbf{I}-P) = \sum_{k=0}^N z^{n-k}(-1)^k\text{tr}\left(\Lambda^k P\right),
				\]
				where $\Lambda^k P$ is the $k^{th}$ exterior power of $P$. Fix $r\ge 1$ and let $Q=P_{r}$ (for simplicity of notation.) Also, without loss of generality, assume for convenience that $\mathcal{L}_{r}=\Sigma^{r}$. Then
				\[
				\text{tr}(Q^\ell) = \sum_{X\in\Sigma^{r}}(Q^\ell)_{XX}= 
				\sum_{X\in\Sigma^{r}}\sum_{Y_1,\dots,Y_{\ell-1}}Q_{XY_1}Q_{Y_1Y_2}\dots Q_{Y_{\ell-1}X}.
				\]
				Let $X=x_1\dots x_r$ and $Y_i=y_{i,1}\dots Y_{i,r}$ for $1\le i\le \ell-1$.
				
				The product $Q_{XY_1}Q_{Y_1Y_2}\dots Q_{Y_{\ell-1}X}$ is nonzero only if $X*Y_1, Y_1*Y_2,\dots, Y_{\ell-1}*X$ are defined, in which case, it equals 
				\[
				P_{t_2(Y_1)}\dots P_{t_2(Y_{\ell-1})}P_{t_2(X)}=P_{x_ry_{1,r}}\dots P_{y_{\ell-2,r}y_{\ell-1,r}}P_{y_{\ell-1,r}x_r}.
				\]
				Therefore we have that $\text{tr}(Q^\ell)=\text{tr}(P^\ell)$, for all $\ell\ge 1$. By the definition of exterior power, we have $\text{tr}\left(\Lambda^\ell Q\right)=\text{tr}\left(\Lambda^\ell P\right)$. Hence
				\begin{eqnarray*} 
					\text{det}(z\mathbf{I}-Q)&=&\sum_{k=0}^{N^r} z^{N^r-k}(-1)^k\text{tr}\left(\Lambda^k Q\right) = \sum_{k=0}^{N^n} z^{N^r-k}(-1)^k\text{tr}\left(\Lambda^k P\right) \\
					&=& \sum_{k=0}^{N} z^{N^r-k}(-1)^k\text{tr}\left(\Lambda^k P\right)= z^{N^r-N}\text{det}(z\mathbf{I}-P).
				\end{eqnarray*} 
				We have used the fact that $\Lambda^k(P)=0$ for $k>N$. Hence we have
				\[
				\text{det}(z\mathbf{I}-P_{r})=z^{N^{r}-N}\text{det}(z\mathbf{I}-P).
				\]
				Therefore, \[
				\text{det}(\mathbf{I}-zP_r)=z^{N^r}\text{det}(z^{-1}I-P_r)=z^N\text{det}(z^{-1}I-P)=\text{det}(\mathbf{I}-zP).
				\]
			\end{proof}
			\noindent We now add the following two lemmas proofs of which are skipped. 
			\begin{lemma}\label{lemma:det-adj}
				If $M$ is a matrix of size $N$ with
				\[
				\text{det}(\mathbf{I}-zM)= \sum_{i=0}^N \alpha_{i}z^{N-i},
				\]
				then
				\[
				\text{adj}(\mathbf{I}-zM)=\sum_{0\le j+k\le N-1} \alpha_{j+k+1}z^{N-1-j}M^k.
				\]
			\end{lemma}

			\begin{lemma}\label{lemma:powerP}
				For $u=u_1\dots u_r$, $X=u_1\dots u_{r-1}$, and $Y=u_2\dots u_r$, 
				\[
				P_{u_{r-1}u_r}(P_{r-1}^k)_{YX} = \begin{cases}
					c_{k+1,u,u}\delta_{k+1,u}, & 0\le k\le r-2 \\
					\delta_u P_{u_ru_1}^{k-r+2}, & k\ge r-1.
				\end{cases}
				\]
			\end{lemma}
			We restate Theorem~\ref{thm:relate} below with a proof. In the following, we denote $D(z)=\text{det}(\mathbf{I}-zP)$, $A(z)=\text{adj}(\mathbf{I}-zP)$, $D_r(z)=\text{det}(\mathbf{I}-zP_r)$, $A_r(z)=\text{adj}(\mathbf{I}-zP_r)$, for $r\ge 2$. By Lemma~\ref{lemma:det-rel}, $D_r(z)=D(z)$.
			
			\begin{theorem}
				For $u=u_1\dots u_{r}$ and $\G=\{u\}$, 
				\begin{eqnarray}\label{eq:cpfu}
					\text{det}(\mathbf{I}-zB_\G\circ P_{r-1})&=&f_u(z).
				\end{eqnarray}
				Equivalently,
				\begin{eqnarray}\label{eq:toprove}
					(\tilde{\tau}_u(z)-1)D(z)+z^{r-1}\delta_u A(z)_{u_r u_1} &=& zP_{u_{r-1}u_r}A_{r-1}(z)_{YX},
				\end{eqnarray}
				where $X=u_1\dots u_{r-1}$ and $Y=u_2\dots u_{r}$.
			\end{theorem}
			
			\begin{proof}
				Since $B_\G\circ P_{r-1}=P_{r-1}-(P_{r-1})_{XY}\mathbf{e}_X\mathbf{e}_Y^T$, the characteristic polynomial of $B_\G\circ P_{r-1}$ is given by
				\begin{eqnarray*}
					\text{det}(\mathbf{I}-zB_\G\circ P_{r-1}) &=& \text{det}(\mathbf{I}-zP_{r-1}+zP_{u_{r-1}u_r}\mathbf{e}_X\mathbf{e}_Y^T)\\
					&=& D_{r-1}(z)+zP_{u_{r-1}u_r}(\text{adj}(z\mathbf{I}-P_{r-1}))_{YX}\\
					&=& D(z)+zP_{u_{r-1}u_r}A_{r-1}(z)_{YX}.
				\end{eqnarray*}
				We have used Lemma~\ref{lemma:det-rel} in the last step. Thus~\eqref{eq:cpfu} and~\eqref{eq:toprove} are equivalent. \\
				Suppose the characteristic polynomial of the matrix $P$ is $\sum_{i=0}^N \beta_{i}z^{i}$. Then $D(z)=\sum_{i=0}^N \beta_{i}z^{N-i}$. Let $D_{r-1}(z)=\sum_{i=0}^{N^{r-1}}\alpha_{i}z^{N^{r-1}-i}$. By Lemma~\ref{lemma:det-rel}, we get
				\[
				\alpha_i=\begin{cases}
					0, & i\le N^{r-1}-N-1,\\
					\beta_{i-N^{r-1}+N}, &  N^{r-1}-N\le i\le N^{r-1}.
				\end{cases}	
				\]
				Hence by Lemma~\ref{lemma:det-adj},
				\begin{eqnarray*}
					zP_{u_{r-1}u_r}A_{r-1}(z)_{YX} &=& zP_{u_{r-1}u_r}\sum_{0\le j+k\le N^{r-1}-1} \alpha_{j+k+1}z^{N^{r-1}-1-j}(P_{r-1}^k)_{YX} \\
					&=&  P_{u_{r-1}u_r}\sum_{N^{r-1}-N-1\le j+k\le N^{r-1}-1} \alpha_{j+k+1}z^{N^{r-1}-j}(P_{r-1}^k)_{YX},
				\end{eqnarray*}
				since $\alpha_{j+k+1}=0$, for all $0\le j+k\le N^{r-1}-N-2$.
				We first split the summation on the right into two parts; $0\le k\le r-2$ and $k\ge r-1$ in order to use Lemma~\ref{lemma:powerP}. We first look at the part of summation where $0\le k\le r-2$. We get,  
				\begin{eqnarray*}
					&&P_{u_{r-1}u_r}\sum_{k=0}^{r-2}(P_{r-1}^k)_{YX}\sum_{j=N^{r-1}-N-1-k}^{N^{r-1}-1-k}
					\alpha_{j+k+1}z^{N^{r-1}-j}\\
					&& = \sum_{k=0}^{r-2}c_{k+1,u,u}\delta_{k+1,u}\sum_{j=N^{r-1}-N-1-k}^{N^{r-1}-1-k}\beta_{j+k+1-N^{r-1}+N}z^{N^{r-1}-j}\\
					&& = \sum_{k=0}^{r-2}c_{k+1,u,u}\delta_{k+1,u}z^{k+1}\sum_{J=0}^{N}\beta_{J}z^{N-J}\\
					&& =D(z) \left(\tau_u(z)-1\right)\\
					&& =D(z) \left(\tilde{\tau}_u(z)-1\right)+c_{r-1,u,u}\delta_uz^{r-1}D(z),
				\end{eqnarray*}
				using the change of variable $J=j-N^{r-1}+N+1+k$. 
				Thus to obtain the required identity~\eqref{eq:toprove}, we are left to show that the second summation for $k\ge r-1$ satisfies the following:
				\begin{eqnarray*}
					&&P_{u_{r-1}u_r}\sum_{\substack{N^{r-1}-N-1\le j+k\le N^{r-1}-1\\ k\ge r-1}} \alpha_{j+k+1}z^{N^{r-1}-j}(P_{r-1}^k)_{YX}\\
					&& =z^{r-1}\delta_u A(z)_{u_ru_1}-c_{r-1,u,u}\delta_uz^{r-1}D(z).      
				\end{eqnarray*}
				Using Lemma~\ref{lemma:powerP}, this is equivalent to proving
				\begin{eqnarray*}
					\sum_{\substack{N^{r-1}-N-1\le j+k\le N^{r-1}-1\\ k\ge r-1}} \alpha_{j+k+1}z^{N^{r-1}-j-r+1}P_{u_ru_1}^{k-r+2}&=& A(z)_{u_ru_1}-c_{r-1,u,u}D(z).      
				\end{eqnarray*}
				Since $c_{r-1,u,u}=\chi_{u_1,u_r}$, and $u_1,u_r$ are arbitrary, we prove a more general statement for matrices: 
				\begin{eqnarray*}
					\sum_{\substack{N^{r-1}-N-1\le j+k\le N^{r-1}-1\\ k\ge r-1}} \alpha_{j+k+1}z^{N^{r-1}-j-r+1}P^{k-r+2}&=& A(z)-D(z)I.      
				\end{eqnarray*}
				We break the summation on the left into two parts:
				\[\sum_{j=0}^{N^{r-1}-N-r}\sum_{k=N^{r-1}-N-1-j}^{N^{r-1}-1-j}+\sum_{j=N^{r-1}-N-r+1}^{N^{r-1}-r}\sum_{k=r-1}^{N^{r-1}-1-j}.
				\]
				We will prove that the first part is equal to the zero matrix $\mathbf{O}$. Indeed we have for each $0\le j\le N^{r-1}-N-r$,
				\[\sum_{k=N^{r-1}-N-1-j }^{N^{r-1}-1-j} \alpha_{j+k+1}P^{k-r+2} =P^{N^r-N-r+1-j}\sum_{K=0}^N \beta_K P^K = \mathbf{O},
				\]
				using the change of variable $K=j+k+1-N^{r-1}+N$. \\
				Now consider the second part:
				\begin{eqnarray*}
					&& \sum_{j=N^{r-1}-N-r+1}^{N^{r-1}-r}\sum_{k=r-1}^{N^{r-1}-1-j}z^{N^{r-1}-r-j+1} \alpha_{j+k+1}P^{k-r+2} \\
					&&  \ \ = \sum_{J=0}^{N-1}\sum_{K=0}^{N-1-J}\beta_{J+K+1}z^{N-J}P^{K+1}\\
					&&    \ \ = zA(z)P= A(z) - D(z)I,
				\end{eqnarray*}
				using the change of variables $J=j-(N^{r-1}-N-r+1)$ and $K=k-r+1$ and Lemma~\ref{lemma:det-adj}. Thus we have proved the required identity.
			\end{proof}

			\section{Concluding remarks}
			This paper gives two methods to compute the escape rate into a Markov hole. The first method describes escape rate in terms of the spectral radius of a specific matrix obtained using the higher block representation of the shift. The other method gives escape rate in terms of the smallest real positive pole of a certain rational function obtained using recurrence relations. Each method has limitations and advantages, as discussed in Section~\ref{sec:advantages}. We make a few observations now. For a shift on two symbols, for two holes corresponding to words $u=u_1\dots u_r$ and $v=v_1\dots v_r$ of the same length and with $\delta_u P_{u_ru_1}=\delta_v P_{v_rv_1}$, by a result in~\cite{BCL}, if $u$ and $v$ are both prime, then $\rho(C_u)=\rho(C_v)$. This fails to hold true if the symbol set has a size of at least 3, as illustrated below. Let 
			\[
			P=\begin{pmatrix}
				0.35 & 0.3 & 0.35\\ 0.3 & 0.4 & 0.3 \\ 0.35& 0.3 & 0.35
			\end{pmatrix}
			\]
			be a doubly stochastic matrix defining the Markov measure on the subshift. Consider two holes corresponding to prime words $u=112$ and $v=321$. Simple calculations show that, the escape rate into the holes are $\rho(C_u)=\ln (1.03786)$ and $\rho(C_v)=\ln (1.03189)$. Here $\delta_uP_{21}=\delta_v P_{13}$, but escape rates are different.
			
			Moreover, by a result in~\cite{BCL}, if the subshift has a product measure, for two holes corresponding to two prime words $u$ and $v$ of the same length and same measure, we have $\rho(C_u)=\rho(C_v)$. This fails if the underlying Markov measure is not a product measure. For instance, let 
			\[
			P'=\begin{pmatrix}
				0.1 & 0.25 & 0.3 & 0.35\\ 0.25&0.15&0.4&0.2\\ 0.3& 0.4 & .05 & 0.25\\ 0.35 & 0.2 & 0.25 & 0.2
			\end{pmatrix}\] 
			be the doubly stochastic matrix defining the Markov measure $\mu=\mu_P$ on the subshift. Consider two holes corresponding to prime words $u=12$ and $v=34$. Here $\mu(u)=\mu(v)$ but $\rho(C_u)\ne\rho(C_v)$.
			
			These two examples above demonstrate that the patterns and relations for escape rates of various holes observed in the simple situations in a shift on two symbols or with a product measure fail in general. Therefore, it becomes extremely challenging to gauge any possible pattern in the case of the symbol set of large size and for any Markov measure. Supported by numerical results and preliminary analysis, we conjecture the following. Suppose $u$ and $v$ are two prime words of the same length. If $\mu(u) < \mu(v)$, then $\rho(C_u) < \rho(C_v)$. The following example shows that if we let go of the assumption that $u$ and $v$ are both prime words, then the conjecture fails. Let the Markov measure on the subshift be defined by the doubly stochastic matrix $P$ as above. Consider two holes corresponding to words $u=12$ and $v=22$. The measures of these holes satisfy $\mu(u)<\mu(v)$. However, the escape rates into the holes satisfy $\rho(C_u)>\rho(C_v)$. In this case, $u$ is prime, and $v$ is not.

			\bibliographystyle{plain} 
			\bibliography{ref}
		\end{document}